\documentclass[11pt,reqno]{amsart} 
\usepackage[english]{babel}
\usepackage[left=1 in, right=1 in, top=1.5 in, bottom=1.5 in]{geometry}
\usepackage{xcolor}
\usepackage{amsmath}
\usepackage{comment}
\usepackage{amssymb}
\usepackage{mathtools}
\usepackage[utf8]{inputenc}
\usepackage{color}
\usepackage{amsthm}
\usepackage{tikz}
\usetikzlibrary{graphs,arrows.meta}
\usepackage[shortlabels]{enumitem}
\usepackage{graphicx}
\usepackage{caption}
\usepackage{subcaption}
\usepackage{hyperref}
\usepackage{xcolor}
\usepackage{xr}
\usepackage[capitalise,nameinlink]{cleveref}
\usepackage{float}

\numberwithin{equation}{section}

\newcommand{\vanish}[1]{}
\newcommand{\IPF}{\mathrm{IPF}}
\newcommand{\IPPF}{\mathrm{IPPF}}
\newcommand{\PF}{\mathrm{PF}}
\newcommand{\PPF}{\mathrm{PPF}}
\newcommand{\wt}{\mathrm{wt}}

\DeclareMathOperator{\Lc}{\mathcal{L}}
\DeclareMathOperator{\N}{\mathbb{N}}
 
\theoremstyle{definition}
\newtheorem{definition}{Definition}[section]
\newtheorem{proposition}[definition]{Proposition}
\newtheorem{lemma}[definition]{Lemma}
\newtheorem{theorem}[definition]{Theorem}

\newtheorem{example}[definition]{Example}

\theoremstyle{remark}
\newtheorem{remark}{Remark}[section]

\newcommand{\bsy}{\boldsymbol}

\newcommand{\tcr}[1]{\textcolor{red}{#1}}
\newcommand{\tcb}[1]{\textcolor{blue}{#1}}

 \usepackage[colorinlistoftodos]{todonotes}

\title{Primeness of generalized parking functions}

\author[Armon]{Sam Armon}
\author[Beckford]{Joanne Beckford}
\author[Hanson]{Dillon Hanson}
\author[Krawzik]{Naomi Krawzik}
\author[Mandelshtam]{Olya Mandelshtam}
\author[Martinez]{Lucy Martinez}
\author[Yan]{Catherine Yan}

\address[Armon]{Department of Mathematics, University of Southern California, Los Angeles, CA 90007} 
\email{armon@usc.edu}

\address[Beckford]{Department of Mathematics, Bryn Mawr College, Bryn Mawr, PA, 19010} 
\email{jbeckford@brynmawr.edu}

\address[Hanson]{Department of Mathematics, Jacksonville University, Jacksonville, FL 32211}
\email{dhanson3@ju.edu}

\address[Krawzik]{Department of Mathematics and Statistics, Sam Houston State University, Huntsville, TX 77340} 
\email{krawzik@shsu.edu} 

\address[Mandelshtam]{Department of Combinatorics and Optimization, University of Waterloo, Waterloo, ON, Canada}
\email{omandels@uwaterloo.ca} 

\address[Martinez]{Department of Mathematics, Rutgers University, Piscataway, NJ 08854}
\email{lm1154@scarletmail.rutgers.edu}

\address[Yan]{Department of Mathematics, Texas A\&M University, College Station, TX 77843}
\email{huafei-yan@tamu.edu}

\date{June 2024} 

\begin{document}

\begin{abstract}
Classical parking functions are a generalization of permutations  that appear in many combinatorial structures. Prime parking functions are indecomposable components  such that any classical parking function can be uniquely described as a direct sum of prime ones. 
In this article, we extend the notion of primeness to three generalizations of classical parking functions: vector parking functions, $(p,q)$-parking functions, and two-dimensional vector parking functions. We study their enumeration by obtaining explicit formulas for the number of prime vector parking functions when the vector is an arithmetic progression, prime $(p,q)$-parking functions, and prime two-dimensional vector parking functions when the weight matrix is an affine transformation of the coordinates.
\end{abstract}

\maketitle 

\section{Introduction}

Let $\mathbb{N}=\{0,1,2,\ldots\}$ and set $\mathbb{N}_n=\{0,1,\ldots, n-1\}$ for $n\geq 1$. A \textit{classical parking function,} or simply a \textit{parking function}, of length $n$ is a sequence $\bsy{a} = (a_0,a_1, \ldots, a_{n-1})\in (\mathbb{N}_n)^n$  
 whose weakly increasing rearrangement, denoted $(a_{(0)}, a_{(1)},\ldots,a_{(n-1)})$, satisfies $a_{(i)} \leq i$ for all $i \in \mathbb{N}_n$. We call this weakly increasing sequence the set of \emph{order statistics} of $\bsy{a}$, where each $a_{(i)}$ is the $i$-th order statistic of $\bsy{a}$. 

Parking functions can be described by a deterministic parking process, as follows. Consider a parking lot with $n$ parking spots on a one-way street labeled by $\mathbb{N}_n$. A queue of $n$ cars enters the lot one by one, with car $i$ having a preferred spot $a_i$, which we call its \emph{parking preference}. For each $i\in\mathbb{N}_n$, car $i$ drives to its preferred spot $a_i$ and attempts to park. If the spot is not available, the car continues on and parks in the next available spot, if one exists. If there is no available spot, the car will exit the lot. We call this the \emph{parking rule}, and the list of preferences {$\bsy{a}=(a_0,a_1, \ldots, a_{n-1})$} is called a \emph{parking function} if all cars are able to park in the lot under the parking rule. It is well-known that parking functions {of length $n$} are enumerated by $(n+1)^{n-1}$, see \cite{KonheimWeiss,Pyke}.  

The inequalities satisfied by the order statistics can be restated as
\begin{equation}\label{eq: classical case inequality definition}
    \#\{j:a_j\leq i\}\geq i+1\qquad \mbox{for}\quad {i=0,\dots,n-1,} 
\end{equation} 
that is, $\bsy{a}$ is a parking function if the number of $a_j \leq i$ is weakly bounded below by $i+1$ for each $i$.

It is convenient to represent parking functions as lattice paths with labeled vertical steps. For $n,m\geq 0$, define $\mathcal{L}(n,m)$ to be the set of lattice paths from $(0,0)$ to $(n,m)$ consisting of north-steps (denoted by N) and east-steps (denoted by E). For a sequence $\bsy{a}\in(\mathbb{N}_n)^n$, define the lattice path $L_{\bsy{a}}\in\mathcal{L}(n,n)$ to be the unique path whose $i$-th \emph{vertical} edge has $x$-coordinate $a_{(i)}$ for each $i\in\mathbb{N}_n$. Then $\bsy{a}$ is a parking function if and only if $L_{\bsy{a}}$ is a \textit{Catalan path}, one which is weakly bounded below by the diagonal $y=x$. Then, the vertical edges of $L_{\bsy{a}}$ corresponding to a parking function $\bsy{a}$ are labeled by $\mathbb{N}_n$ so that the $x$-coordinate of the edge with the label $j$ is equal to $a_j$, and labels with the same $x$-coordinate increase from bottom to top.

An \emph{increasing parking function} $\bsy{a}=(a_0,\ldots,a_{n-1})$ satisfies $a_0\leq a_1\leq\cdots\leq a_{n-1}$, and therefore, $a_i=a_{(i)}$. The set of increasing parking functions is in bijection with the set of Catalan paths. See \cref{fig:PF paths} for an example. Consequently, the increasing parking functions of length $n$ {are enumerated} by the Catalan number $C_n=\frac{1}{n+1}\binom{2n}{n}$.

\begin{figure}[H]
\begin{center}
\begin{tikzpicture}[scale=0.8]

\begin{scope}[xshift=10cm]
\draw[step=1cm, thick, dotted] (0,0) grid (4,4);
\node at (-.5,.5) {\tcr{$\mathbf{0}$}}; 
\node at (-.5,1.5) {\tcr{$\mathbf{0}$}}; 
\node at (-.5,2.5) {\tcr{$\mathbf{2}$}};
\node at (-.5,3.5 ) {\tcr{$\mathbf{3}$}}; 

\node at (-1.5,4) {(b)};

\foreach \i in {0,1,2,3,4}
{
\node at (4.5,\i) {\tiny \i};
\node at (\i,4.5) {\tiny \i};
}

\node[red] at (0.2,0.5) {\tiny 1};
\node[red] at (0.2,1.5) {\tiny 3};
\node[red] at (2.2,2.5) {\tiny 0};
\node[red] at (3.2,3.5) {\tiny 2};

\draw[red, very thick] (0,0)--(0,2)--(2,2)--(2,3)--(3,3)--(3,4)--(4,4);

\draw[black] (0,0)--(4,4); 
\end{scope}

\begin{scope}[xshift=0cm]
\draw[step=1cm, thick, dotted] (0,0) grid (4,4);
\node at (-.5,.5) {\tcr{$\mathbf{0}$}}; 
\node at (-.5,1.5) {\tcr{$\mathbf{0}$}}; 
\node at (-.5,2.5) {\tcr{$\mathbf{2}$}};
\node at (-.5,3.5 ) {\tcr{$\mathbf{3}$}};

\node at (-1.5,4) {(a)};

\foreach \i in {0,1,2,3,4}
{
\node at (4.5,\i) {\tiny \i};
\node at (\i,4.5) {\tiny \i};
}

\draw[red, very thick] (0,0)--(0,2)--(2,2)--(2,3)--(3,3)--(3,4)--(4,4);  

\draw[black] (0,0)--(4,4);  
\end{scope}
\end{tikzpicture}
\caption{
(a) The increasing parking function $\bsy{a}=(0,0,2,3)$ and parking function $\bsy{b}=(2,0,3,0)$ correspond to the same Catalan path $L_{\bsy{a}}=L_{\bsy{b}}$. 
(b) The parking function $\bsy{b}$ is represented by adding labels $j\in\mathbb{N}_4$ to the vertical edges of $L_{\bsy{b}}$ such that $j$ has $x$-coordinate $b_{j}$. 
}\label{fig:PF paths}  
\end{center} 
\end{figure}
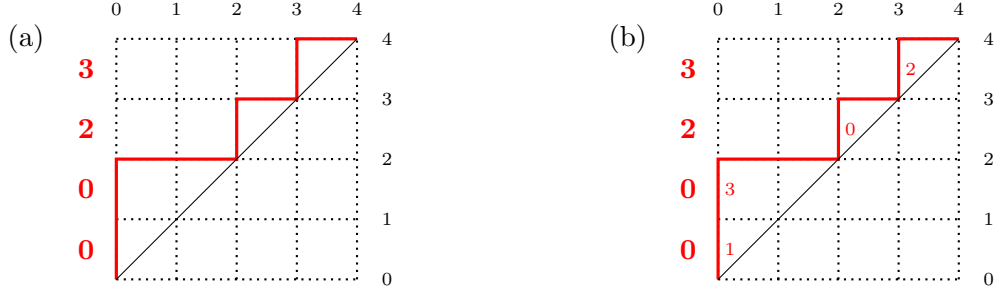

Many mathematical objects can be uniquely decomposed into prime components, a process commonly referred to as prime decomposition. For instance, consider a Catalan path $P\in\mathcal{L}(n,n)$: There is a unique way to split this path into nontrivial components $P_1,P_2,\ldots,P_k$ for some $k\geq 1$ such that each  $P_i$ touches the line $y=x$ only at its endpoints. We represent this split as a direct sum $P=P_1\oplus\cdots\oplus P_k$, which should be interpreted as concatenating the endpoints of the $P_i$'s from left to right. Furthermore, each $P_i$ can be written as $\{{N}\}\oplus P'_i\oplus \{{E}\}$, where $P'_i$ is itself a Catalan path. We call the components $P_i$ the \emph{prime Catalan paths}, so that any Catalan path is either empty or can be written uniquely as a finite sequence of prime Catalan paths. This yields the following decomposition for the class of Catalan paths $\mathcal{C}$ in terms of the class of prime Catalan paths $\mathcal{P}$:
\[
\mathcal{C}=\epsilon \cup \mathcal{P}^\star,\qquad \mathcal{P}=\{{N}\}\oplus\mathcal{C}\oplus\{{E}\},
\]
where $\epsilon$ is the trivial (empty) path of size 0, and the notation $\mathcal{P}^\star$ indicates the concatenation of any finite number of objects in $\mathcal{P}$. 

It is natural, then, to characterize the set of classical prime parking functions as those whose corresponding Catalan paths are prime.  
Indeed, let $\bsy{a_1}, \bsy{a_2}, \dots, \bsy{a_k}$ be a sequence of prime parking functions, where 
for each $i$, $L_{\bsy{a_i}}$ touches the diagonal $y=x$ only at its endpoints.  
Assume the length of $\bsy{a_i}$ is $d_i$.  For a sequence $\bsy{x}=(x_1, x_2,  \dots,x_k)$ and $r\in\mathbb{N}$, define $\bsy{x}+r=(x_1+r, x_2+r, \dots, x_k+r)$, and let 
$\bsy{b_i}=\bsy{a_i}+d_1+d_2+\cdots +d_{i-1}$ for $i=1,\ldots,k$.  
Then any shuffle of the sequences $\bsy{b_1}, \bsy{b_2}, \dots, \bsy{b_k}$ is a parking function. 
Assume $\bsy{a}$ is such a shuffle and let $B_i$ be the set of positions of entries of $\bsy{b_i}$ in $\bsy{a}$. 
Then in the labeled Catalan path representation of $\bsy{a}$, we have that  $L_{\bsy{a}}=L_{\bsy{a_1}}\oplus\cdots\oplus L_{\bsy{a_k}}$, 
and the labels on $L_{\bsy{a}}$ are obtained by replacing the labels of vertical edges in $L_{\bsy{a_i}}$ with the elements in $B_i$, 
where label $j$ is replaced by the $j$-th  order statistic of $B_i$.  

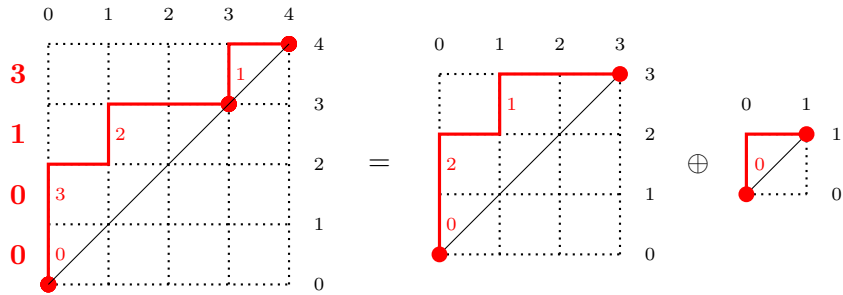
\begin{figure}[H]
\begin{center}
\begin{tikzpicture}[scale=0.8]
\begin{scope}[xshift=0cm]
\draw[step=1cm, thick, dotted] (0,0) grid (4,4);

\node at (-.5,.5) {\tcr{$\mathbf{0}$}}; 
\node at (-.5,1.5) {\tcr{$\mathbf{0}$}}; 
\node at (-.5,2.5) {\tcr{$\mathbf{1}$}};
\node at (-.5,3.5) {\tcr{$\mathbf{3}$}}; 
 
\foreach \i in {0,1,2,3,4}
{
\node at (4.5,\i) {\tiny \i};
\node at (\i,4.5) {\tiny \i};

\foreach \u\i in {0/0,3/3,4/4}
{
\filldraw[red,very thick] (\u,\i) circle (3pt);
}
}

\draw[red, very thick] (0,0)--(0,2)--(1,2)--(1,3)--(3,3)--(3,4)--(4,4);  

\draw[black] (0,0)--(4,4); 

\node[red] at (0.2,0.5) {\tiny 0};
\node[red] at (0.2,1.5) {\tiny 3};
\node[red] at (1.2,2.5) {\tiny 2};
\node[red] at (3.2,3.5) {\tiny 1};


\node at (5.5,2) {$=$};
\draw[step=1cm, thick, dotted,xshift=6.5cm,yshift=.5cm] (0,0) grid (3,3);
\draw[black,xshift=6.5cm,yshift=.5cm] (0,0)--(3,3); 
\draw[red, very thick,xshift=6.5cm,yshift=.5cm] (0,0)--(0,2)--(1,2)--(1,3)--(3,3);
\node[red,xshift=5.2cm,yshift=0.4cm] at (0.2,0.5) {\tiny 0};
\node[red,xshift=5.2cm,yshift=0.4cm] at (0.2,1.5) {\tiny 2};
\node[red,xshift=5.2cm,yshift=0.4cm] at (1.2,2.5) {\tiny 1};

\foreach \i in {0,1,2,3}
{
\node at (10,\i+0.5) {\tiny \i};
\node at (\i+6.5,4) {\tiny \i};
}

\filldraw[red,very thick,xshift=6.5cm,yshift=.5cm] (0,0) circle (3pt);
\filldraw[red,very thick,xshift=6.5cm,yshift=.5cm] (3,3) circle (3pt);

\draw[step=1cm, thick, dotted,xshift=11.6cm,yshift=1.5cm] (0,0) grid (1,1);
\draw[black,xshift=11.6cm,yshift=1.5cm] (0,0)--(1,1); 

\node at (10.8,2) {$\oplus$};
\draw[red, very thick,xshift=11.6cm,yshift=1.5cm] (0,0)--(0,1)--(1,1); 
\node[red,xshift=9.3cm,yshift=1.2cm] at (0.2,0.5) {\tiny 0};

\foreach \i in {0,1}
{
\node at (13.1,\i+1.5) {\tiny \i};
\node at (\i+11.6,3) {\tiny \i};
}

\filldraw[red,very thick,xshift=11.6cm,yshift=1.5cm] (0,0) circle (3pt);
\filldraw[red,very thick,xshift=11.6cm,yshift=1.5cm] (1,1) circle (3pt);

\end{scope}

\end{tikzpicture}
\caption{
The parking function $\bsy{a}=(0,3,1,0)$ can be decomposed as $\bsy{a}=\bsy{a_1}\oplus\bsy{a_2}$ with $\bsy{a_1}=(0,1,0)$ and $\bsy{a_2}=(0)$. Note that $\bsy{b_1}=\bsy{a_1}=(0,1,0)$ and $\bsy{b_2}=\bsy{a_2}+3=(3)$, and $\bsy{a}$ is a shuffle of $\bsy{b_1}$ and $\bsy{b_2}$. Here, $B_1=\{0,2, 3\}$ and $B_2=\{1\}$,  where 
$B_i$ is the set of positions of $\bsy{b_i}$ in $\bsy{a}$, for $i=1, 2$. }\label{fig:prime}
\end{center} 
\end{figure}

Conversely, given any parking function $\bsy{a}$, one can uniquely reconstruct the prime components $\bsy{a_1}, \bsy{a_2}, \dots, \bsy{a_k}$ by taking the prime decomposition of the lattice path $L_{\bsy{a}}=P_1\oplus\cdots\oplus P_k$, and, for $i=1,2,\ldots,k$, associating a parking function $\bsy{a_i}$ to the prime segment $P_i$ with edges relabeled with the label set $\mathbb{N}_{|\bsy{a_i}|}$. 
In this sense, we write $\bsy{a}$ as a direct (shuffle) sum    $\bsy{a}=\bsy{a_1}\oplus\cdots\oplus\bsy{a_k}$.  
See \cref{fig:prime} for an example of such a decomposition.

\begin{definition}\label{def:classical prime}
A sequence $\bsy{a}=(a_0,\ldots,a_{n-1})\in\mathbb{N}^n$ is a prime parking function if
\[\#\{j:a_j\leq i\}>i+1\qquad \mbox{for}\  i=0,\dots,n-2.
\]  
\end{definition}
Translating back to the parking process, a parking function of length $n$ is said to be \emph{prime} if at least $i+2$ cars want to park in spots $0,\dots,i$ for $0\leq i \leq n-2$.
We denote the set of classical parking functions of length $n$ by $\PF(n)$ and the set of prime parking functions of length $n$ by $\PPF(n)$. The following lemma provides a useful equivalent characterization.

\begin{lemma}\label{lem:classical prime equivalent}
Let $\bsy{a}=(a_0,\ldots,a_{n-1})\in\mathbb{N}^n$. Then $\bsy{a}\in\PPF(n)$ if and only if removing a $0$ from $\bsy{a}$ yields a parking function of length $n-1$.
\end{lemma}
\begin{proof}
Suppose $\bsy{a}\in\PPF(n)$. Observe that the first inequality of \cref{def:classical prime} is $\#\{j:a_j\leq 0\}> 1$, so $0\in\bsy{a}$. Let $\bsy{a'}=(a_0',\ldots,a_{n-1}')$ be any sequence obtained by removing a $0$ from $\bsy{a}$. For $i=0,\dots,n-2$, we have $\#\{j:a'_j\leq i\}=\#\{j:a_j\leq i\}-1 \geq i+1$, so $\bsy{a'}\in\PF(n-1)$. Reversing the argument gives the other direction.
\end{proof}

Prime parking functions have appealing enumerative properties, making them intriguing objects to study in their own right. The number of prime parking functions of length $n$ is $(n-1)^{n-1}$ (see, for instance, \cite[Exercise 5.49(f)]{EC2}), and the number of prime increasing parking functions is $C_{n-1}$.

The objective of this paper is to extend and characterize primeness to various families of generalized parking functions. 
The first family is the set of vector parking functions, defined as sequences whose order statistics are bounded by a vector $\bsy{u}$.
Vector parking functions are related to empirical distributions, parking polytopes, 
and Gon\v{c}arov polynomials \cite{KungYan, PitmanStanley}.  The second family is the set of $(p,q)$-parking functions, a two dimensional generalization introduced by Cori and Poulalhon \cite{CP02}, which can be interpreted as
recurrent configurations in the sandpile model for complete bipartite graphs with an extra root. In addition, 
we discuss primeness for a common generalization of these two families--the two dimensional vector parking functions with weight matrix $\bsy{U}$. 

Our paper is structured as follows: In \cref{section:definitions}, we introduce and recall results for vector parking functions, $(p,q)$-parking functions, and two-dimensional vector parking functions. 
In \cref{section:3}, we propose a definition for prime $\bsy{u}$-parking functions that is compatible with the concept of prime decomposition. We  give closed formulas for the enumeration of prime $\bsy{u}$-parking functions when $\bsy{u}$ is an arithmetic sequence. 
In \cref{section:pq parking functions}, we propose an analogous definition for prime $(p,q)$-parking functions and provide closed formulas for their enumeration. 
{Additionally,} we generalize the results to two-dimensional vector parking functions with weight matrix $\bsy{U}$ defined by an affine transformation of the coordinates.  Finally, we list some open questions in \cref{section:final}. 

\section{Definitions}\label{section:definitions}

\subsection{Vector parking functions}

We now consider a parking process where each parking spot can have a nonnegative capacity. Suppose the parking spots are labeled by $\mathbb{N}$, and let $\bsy{u}=(u_0,\ldots,u_{n-1})$ with $1\leq u_0\leq \cdots\leq u_{n-1}$ encode the capacities of the parking spots, such that the capacity of spot $i$ is the multiplicity of $i+1$ in the vector $\bsy{u}$, notated as $m_{i+1}(\bsy{u})$, for $0\leq i\leq n-1$. In particular, $\sum_{i\geq 0} m_{i+1}(\bsy{u})=n$.  

A $\bsy{u}$-parking function can be described via a parking process, as follows. Let $\bsy{a}=(a_0,\ldots,a_{n-1})$ be a list of parking preferences with $a_i\geq 0$ for each $i\in\mathbb{N}_n$. A queue of $n$ cars with parking preferences given by $\bsy{a}$ enters the lot, and each car drives to its preferred spot and attempts to park. A car may park at a spot $j$ if there are fewer than $m_{j+1}(\bsy{u})$ cars already parked there. Otherwise it will attempt the next spot $j+1$, and so on. If there are no available spots past spot $j$, the car must exit the lot. We say that $\bsy{a}$ is a \emph{$\bsy{u}$-(vector) parking function} if all cars are able to park under this parking rule. See \Cref{ex:u parking function parking process example}.

\begin{example}\label{ex:u parking function parking process example}
Let $n=5$, and suppose $\bsy{u}=(1,1,3,3,9)$ and $\bsy{a}=(6,0,1,0,0)$. Then we have $m_1(\bsy{u})=2, m_3(\bsy{u})=2, m_9(\bsy{u})=1$, and $m_{j}(\bsy{u})=0$ for any $j\neq 1,3,9$. Thus there are 3 non-empty spots with total capacity 5: spots 0, 2, and 8 with capacities 2, 2, and 1, respectively. 
The cars with parking preferences given by $\bsy{a}$ then park as follows: Car 0 parks in spot 8, car 1 parks in spot 0, car 2 parks in spot 2, car 3 parks in spot 0, and car 4 parks in spot 2.
\end{example}
Similar to the definition of parking functions that satisfy the inequalities in \cref{eq: classical case inequality definition}, $\bsy{u}$-vector parking functions can also be defined by  a set of inequalities.
\begin{definition} 
\label{vectorpark}  
Let $\bsy{u} = (u_0, \ldots, u_{n-1})$ be a weakly increasing sequence of positive integers. We say a sequence $\bsy{a} = (a_0, \ldots, a_{n-1}) \in \mathbb{N}^n$ with order statistics $(a_{(0)},\ldots,a_{(n-1)})$ is a \textit{$\bsy{u}$-parking function} if $a_{(i)} < u_{i}$ for each $0 \leq i \leq n-1$. In other words, to be a $\bsy{u}$-parking function,  the number of cars that prefer the first $u_i$ spots need to be at least $i+1$: 
\begin{equation}\label{eq:u inequalities}
\#\{j: a_j< u_i\} \geq i+1,\qquad \mbox{for}\quad i=0,\ldots,n-1.
\end{equation}
\end{definition}

Note that the number of effective inequalities 
in  \cref{vectorpark} is the number of distinct values in the vector $\bsy{u}$. 
Moreover, for $i=n-1$, the inequality must be an equality. 

The set of increasing $\bsy{u}$-parking functions corresponds to the set of $\bsy{u}$-parking functions whose entries are weakly increasing. Denote the set of $\bsy{u}$-parking functions and their increasing counterpart by $\PF(\bsy{u})$ and $\IPF(\bsy{u})$, respectively. For readability, we will also sometimes write $\PF_n(\bsy{u})$ and $\IPF_n(\bsy{u})$ to emphasize the length $n$.

For a fixed vector $\bsy{u}=(u_0,\ldots,u_{n-1})$ with $1\leq u_0\leq \cdots\leq u_{n-1}$, let $L_{\bsy{u}}\in\Lc(u_{n-1},n)$ be the unique path with vertical steps having $x$-coordinates $u_0,\ldots,u_{n-1}$. Define the set of lattice paths with \emph{right boundary} $\bsy{u}$ to be the set of paths $P\in \Lc(u_{n-1},n)$ that lie to the left of $L_{\bsy{u}}$ and whose vertical edges never coincide with the vertical edges of $L_{\bsy{u}}$. In particular, if $\bsy{a}$ has order statistics $(a_{(0)},a_{(1)},\ldots,a_{(n-1)})$, then $L_{\bsy{a}}$ has right boundary $\bsy{u}$ if and only if $a_{(i)}<u_i$ for $0\leq i\leq n-1$. From the equivalence with Ineq. \eqref{eq:u inequalities}, we obtain that $\bsy{a}$ is a $\bsy{u}$-parking function if and only if $L_{\bsy{a}}$ has right boundary $\bsy{u}$. See \cref{fig:u-vector}(a) for an example.

 It is immediate that the set of increasing $\bsy{u}$-parking functions is equivalent to the set of lattice paths with right boundary $\bsy{u}$, which has been well-studied, for example, in \cite[Chapter 2]{Mohanty79}.

As in the classical case, $\bsy{u}$-parking functions can be represented by labeled lattice paths with right boundary $\bsy{u}$ by placing labels on the vertical edges to correspond to the order of the entries in the parking function. See \cref{fig:u-vector}(b) for an example.

\begin{figure}[H]
\begin{center}
\begin{tikzpicture}[scale=0.8]

\begin{scope}[xshift=0cm]
\node at (-2,3) {(a)};
\draw[step=1cm, thick, dotted] (0,0) grid (3,3);

\foreach \i in {0,1,2,3}
{
\node at (3.4,\i) {\tiny \i};
\node at (\i,3.3) {\tiny \i};
}

\draw[blue, very thick] (0,0)--(1,0)--(1,2)--(3,2)--(3,3);
\draw[blue, very thick] (1,0)--(1,2); 
\draw[blue, very thick] (3,2)--(3,3);

\end{scope}

\begin{scope}[xshift=9.5cm]
\node at (-2,3) {(b)};
\draw[step=1cm, thick, dotted] (0,0) grid (3,3);

\node at (-0.7, 0.5) {\tcr{$\mathbf{0}$}}; 
\node at (-0.7, 1.5) {\tcr{$\mathbf{0}$}}; 
\node at (-0.7, 2.5) {\tcr{$\mathbf{2}$}};

\node[red] at (0.2,0.5) {\tiny $0$};
\node[red] at (0.2,1.5) {\tiny $2$};
\node[red] at (2.2,2.5) {\tiny $1$};

\foreach \i in {0,1,2,3}
{
\node at (3.4,\i) {\tiny \i};
\node at (\i,3.3) {\tiny \i};
}

\draw[blue,very thick] (0,0)--(1,0)--(1,2);
\draw[blue,very thick,yshift=-0.03cm] (1,2)--(2,2);
\draw[blue,very thick] (2,2)--(3,2)--(3,3);
\draw[blue, very thick] (1,0)--(1,2); 
\draw[blue, very thick] (3,2)--(3,3);
\draw[red, very thick] (0,0)--(0,2)--(1,2);
\draw[red, very thick, yshift=0.03cm] (1,2)--(2,2);
\draw[red,very thick] (2,2)--(2,3)--(3,3); 

\end{scope}

\end{tikzpicture}
\caption{(a) For $\bsy{u}=(1,1,3)$, the lattice path $L_{\bsy{u}}=ENNEEN$ has vertical edges given by $\bsy{u}$. (b) For $\bsy{a}=(0,2,0)$, the labeled lattice path $L_{\bsy{a}}=NNEENE$ has right boundary $\bsy{u}$, confirming that $\bsy{a}$ is a $\bsy{u}$-parking function. 
}\label{fig:u-vector}
\end{center} 
\end{figure}
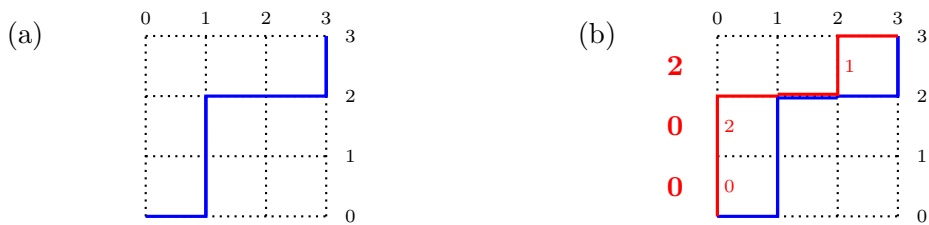

For general $\bsy{u}$, the number of $\bsy{u}$-parking functions is given by a determinantal formula \cite{KungYan, PitmanStanley}, which is not easy to evaluate. When $\bsy{u}$ is an arithmetic progression, it is well-known that the  formulas have a simple form.

\begin{theorem}\label{Theorem:classicalUPF}
Let $\bsy{u}$ be given by the arithmetic progression $u_i=a+bi$ with $a, b \in \mathbb{N}$. 
\begin{enumerate}
    \item The number of $\bsy{u}$-parking functions is $\#PF(\bsy{u})=a(a+bn)^{n-1}$. 
    \item  The number of increasing $\bsy{u}$-parking functions is 
        \[
        \#IPF(\bsy{u})=\frac{a}{a+n(b+1)} \binom{ a+n(b+1)}{n}. 
        \]
    \end{enumerate}
\end{theorem}

When $a=b=1$, \cref{Theorem:classicalUPF} recovers the enumeration for the classical case.

\subsection{$(p,q)$-parking functions}

Cori and Poulalhon introduced $(p,q)$-parking functions for fixed positive integers $p$ and $q$ \cite{CP02}. We will follow an equivalent formulation by Snider and Yan \cite{SY22}, which allows us to extend the definition to all $p,q\in\N$.

\begin{definition}[{\cite[Theorem 3.4]{SY22}}]\label{def:pq}
 Let $\bsy{a}=(a_0,\ldots,a_{p-1})$ and $\bsy{b}=(b_0,\ldots,b_{q-1})$ be two vectors in $\N^p\times\N^q$ with $(a_{(0)},\ldots,a_{(p-1)})$ and $(b_{(0)},\ldots,b_{(q-1)})$ the corresponding increasing rearrangements. Let $L_{\bsy{b}}\in\Lc(p,q)$ be the unique path whose vertical steps have $x$-coordinates $b_{(0)},\ldots,b_{(q-1)}$. Let $L^\perp_{\bsy{a}}\in\Lc(p,q)$ be the unique path whose horizontal steps have $y$-coordinates $a_{(0)},\ldots,a_{(p-1)}$. Then we say that $(\bsy{a},\bsy{b})$ is a \emph{$(p,q)$-parking function} if $L_{\bsy{b}}$ is weakly above $L^\perp_{\bsy{a}}$.
\end{definition}

\noindent Note that the symbol ``$\perp$'' is used since $L^\perp_{\bsy{a}}$ could be viewed as the reflection of $L_{\bsy{a}}$ across the line $y=x$. See \cref{fig:pq} for an example. 

As with classical parking functions, an \emph{increasing $(p,q)$-parking function} is a $(p,q)$-parking function whose entries $(a_0,\ldots,a_{p-1})$ and $(b_0,\ldots,b_{q-1})$ are weakly increasing. We denote the sets of $(p,q)$-parking functions by $\PF(p,q)$ and the increasing ones by $\IPF(p,q)$. An increasing $(p,q)$-parking function $(\bsy{i},\bsy{j})$ can be identified with its corresponding pair of lattice paths $L^\perp_{\bsy{i}}$ and $L_{\bsy{j}}$. A $(p,q)$-parking function $(\bsy{a},\bsy{b})$ can be represented by labeling the edges of $L^\perp_{\bsy{a}}$ and $L_{\bsy{b}}$ as follows. The horizontal edges of $L^\perp_{\bsy{a}}$ are labeled with elements of $\N_p$, where edge $i$ has $y$-coordinate $a_i$, with labels increasing from left to right for edges with the same $y$-coordinate. Similarly, the vertical edges of $L_{\bsy{b}}$ are labeled with elements of $\N_q$, where edge $j$ has $x$-coordinate $b_j$, with labels increasing from bottom to top for edges with the same $x$-coordinate. See \cref{ex:pq} for an example. 

\cref{def:pq} is equivalent to the following characterization.

\begin{lemma}
\label{IPFLemma} 
   The pair $(\bsy{a},\bsy{b})\in\N^p\times\N^q$ with corresponding order statistics $(a_{(0)},\ldots,a_{(p-1)})$ and $(b_{(0)},\ldots,b_{(q-1)})$ is a $(p,q)$-parking function if and only if the following hold:
    \begin{itemize}
        \item[i)] $\#\{j:a_j<i+1\}\geq b_{(i)}$
for $i=0,\ldots,q-1$, and
        \item[ii)] $\#\{j:b_j<i+1\}\geq a_{(i)}$
for $i=0,\ldots,p-1$.
    \end{itemize}
\end{lemma}
\begin{proof}
    Let $(\bsy{a},\bsy{b})\in\N^p\times\N^q$ with order statistics $(a_{(0)},\ldots,a_{(p-1)})$ and $(b_{(0)},\ldots,b_{(q-1)})$. First, if $(\bsy{a},\bsy{b})$ is not a $(p,q)$-parking function, then there exist nonnegative integers $c_1 < c_2$ and $r_1 < r_2$ such that $L^\perp_{\bsy{a}}$ passes through $(c_1,r_2)$ and $L_{\bsy{b}}$ passes through $(c_2,r_1)$. This implies that $a_{(c_1)} \ge r_2$ and $b_{(r_1)} \ge c_2$. In turn, this forces
    \[
    \# \{ j : b_j < c_2 \} \le r_1 < r_2 \le a_{(c_1)} \le a_{(c_2-1)},
    \]
    so that condition ii) of the lemma is not satisfied.

    Conversely, if the inequalities in the lemma are not satisfied, assume without loss of generality that there exists some $0 \le i \le q-1$ such that $\#\{ j : a_j < i+1 \} < b_{(i)}$. In particular, this implies that $a_{(b_{(i)}-1)} \ge i+1$, so that $L^\perp_{\bsy{a}}$ passes through the point $(b_{(i)}-1,i+1)$. However, $L_{\bsy{b}}$ passes through the point $(b_{(i)},i)$ by construction, so $L_{\bsy{b}}$ must not lie weakly above $L^\perp_{\bsy{a}}$, i.e. $(\bsy{a},\bsy{b})$ is not a $(p,q)$-parking function. Thus the two conditions are equivalent.
\end{proof}

\begin{example}

\label{ex:pq}
For $(p,q)=(3,4)$, the pair $\bsy{a}=(3,0,3)$ and $\bsy{b}=(1,0,1,0)$ is a $(p,q)$-parking function whose lattice path representations $L^\perp_{\bsy{a}}$ and $L_{\bsy{b}}$ are shown in \cref{fig:pq}(a) with the labels on the edges representing the indices of the corresponding elements.  On the other hand, the pair $\bsy{a'}=(0,3,3)$ and $\bsy{b'}=(0,0,2,2)$ is not a $(p,q)$-parking function since $L^\perp_{\bsy{a'}}$  and $L_{\bsy{b'}}$ cross, as shown in \cref{fig:pq}(b).
\end{example}

\begin{figure}[H]
\begin{center}
\begin{tikzpicture}[scale=0.8]
\node at (-2,4) {(a)};
\draw[step=1cm, thick, dotted] (0,0) grid (3,4);

\node at (.5, -.5) {\tcr{$\mathbf{0}$}}; 
\node at (1.5,-.5) {\tcr{$\mathbf{3}$}}; 
\node at (2.5, -.5) {\tcr{$\mathbf{3}$}}; 
\node at (-0.7, 0.5) {\tcb{$\mathbf{0}$}}; 
\node at (-0.7, 1.5) {\tcb{$\mathbf{0}$}}; 
\node at (-0.7, 2.5) {\tcb{$\mathbf{1}$}}; 
\node at (-0.7, 3.5) {\tcb{$\mathbf{1}$}}; 

\foreach \i in {0,1,2,3,4}
{
\node at (3.5,\i) {\tiny \i};
}
\foreach \i in {0,1,2,3}
{
\node at (\i,4.5) {\tiny \i};
}

\node[red] at (0.5,0.2) {\tiny 1};
\node[red] at (1.5,3.2) {\tiny 0};
\node[red] at (2.5,3.2) {\tiny 2};
\node[blue] at (-0.2,0.5) {\tiny 1};
\node[blue] at (-0.2,1.5) {\tiny 3};
\node[blue] at (.8,2.5) {\tiny 0};
\node[blue] at (.8,3.5) {\tiny 2};

\draw[red, very thick] (0,0)--(1,0)--(1,2);
\draw[red, very thick,xshift=0.03cm] (1,2)--(1,3);
\draw[red, very thick] (1,3)--(3,3)--(3,4); 
\draw[blue, very thick] (0,0)--(0,2)--(1,2);
\draw[blue, very thick,xshift=-0.03cm] (1,2)--(1,3);
\draw[blue, very thick] (1,3)--(1,4)--(3,4); 

\begin{scope}[xshift=10cm]
\node at (-2,4) {(b)};

\draw[step=1cm, thick, dotted] (0,0) grid (3,4);

\node at (.5, -.5) {\tcr{$\mathbf{0}$}}; 
\node at (1.5,-.5) {\tcr{$\mathbf{3}$}}; 
\node at (2.5, -.5) {\tcr{$\mathbf{3}$}}; 
\node at (-0.7, 0.5) {\tcb{$\mathbf{0}$}}; 
\node at (-0.7, 1.5) {\tcb{$\mathbf{0}$}}; 
\node at (-0.7, 2.5) {\tcb{$\mathbf{2}$}}; 
\node at (-0.7, 3.5) {\tcb{$\mathbf{2}$}}; 

\foreach \i in {0,1,2,3,4}
{
\node at (3.5,\i) {\tiny \i};
}
\foreach \i in {0,1,2,3}
{
\node at (\i,4.5) {\tiny \i};
}

\draw[red, very thick] (0,0)--(1,0)--(1,2)--(1,3)--(3,3)--(3,4); 
\draw[blue, very thick] (0,0)--(0,2)--(2,2)--(2,4)--(3,4);    
\end{scope}
\end{tikzpicture}
\caption{
(a) For the pair $\bsy{a}=(3,0,3)$ and $\bsy{b}=(1,0,1,0)$, $L_{\bsy{b}}=NNENNEE$ is weakly above $L^\perp_{\bsy{a}}=ENNNEEN$, confirming that $(\bsy{a},\bsy{b})$ is a $(3,4)$-parking function. (b) For the pair $\bsy{a'}=(0,3,3)$ and $\bsy{b'}=(0,0,2,2)$, $L_{\bsy{b'}}=NNEENNE$ is not weakly above $L^{\perp}_{\bsy{a'}}=ENNNEEN$, and thus $(\bsy{a'},\bsy{b'})$ is not a (3,4)-parking function.
}\label{fig:pq}
\end{center}
\end{figure}
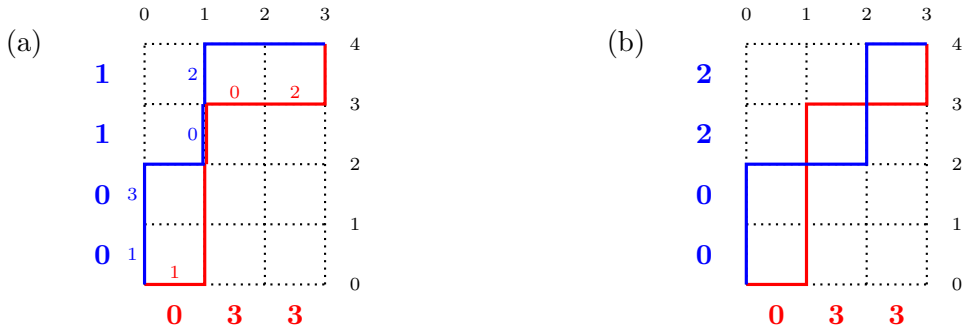

In \cite{CP02}, Cori and {Poulalhon} computed that the number of $(p,q)$-parking functions is 
\[
\#\PF(p,q)=(p+q+1)(p+1)^{q-1}(q+1)^{p-1},
\]
and the number of increasing $(p,q)$-parking functions, with $n=p+q$, is the Narayana number 
\[
\#\IPF(p,q)=\frac{1}{n+1}\binom{n+1}{p}\binom{n+1}{q}.
\] 

{For our discussion, we extend the definition of $(p,q)$-parking functions to the case where $pq=0$. 
Explicitly, when $p=0$ and $q=0$, there is only one $(p,q)$-parking function, namely, $(\emptyset, \emptyset)$, which also counts as increasing. When $p=0$ but $q >0$, there is only one $(p,q)$-parking function, $(\emptyset, (0, \dots, 0))$, which is also increasing. In this case, $L^\perp_\emptyset$ is the unique lattice path in $\Lc(0,q)$ with no horizontal steps --- namely, the path $N^q$ --- so that $L^\perp_\emptyset$ and $L_{(0,\dots,0)}$ coincide. Similarly in the case that $p > 0$ and $q = 0$, we only have the increasing $(p,q)$-parking function $((0,\dots,0),\emptyset)$, and $L^\perp_{(0,\dots,0)} = L_\emptyset = E^p \in \Lc(p,0)$. Note that the above counting formulas still hold for these cases.}

\subsection{Two-dimensional vector parking functions}
As will be seen further on, $(p,q)$-parking functions are a special case of the more general set of two-dimensional vector parking functions, also known as \emph{$\bsy{U}$-parking functions}. 
{$\bsy{U}$-parking functions originated in the study of bivariate Gon\v{c}arov polynomials, which are solutions to the Gon\v{c}arov interpolation problem \cite{KLY14}. Gon\v{c}arov polynomials $\{g_n\}_{n\geq 0}$ provide a natural algebraic tool to enumerate the $\bsy{u}$-parking functions, with the basic relations that $\# \PF(\bsy{u})= g_n(x; x-u_0, \dots, x-u_{n-1})$; see \cite{KungYan}. Similarly, one can consider the Gon\v{c}arov interpolation problem in higher dimensions; in two dimensions, a basis for the set of solutions is given by bivariate Gon\v{c}arov polynomials associated with a certain set of nodes $\bsy{U}\subset \mathbb{N}^2$. Then, the corresponding combinatorial objects are the two-dimensional vector parking functions bounded by entries of the weight matrix given by $\bsy{U}$. The exact definition is stated below.}

 For $p,q \in \mathbb{N}$, let $ \bsy{U} \subset \mathbb{N}^2$ be a set of nodes $\bsy{U} = \{z_{k,\ell}=(u_{k,\ell},v_{k,\ell}) : (0,0) \leq (k,\ell) \leq (p,q) \}$. Define $G_{p,q}(\bsy{U})$ to be the directed graph whose vertices are the lattice points $\{(k,\ell) : (0,0) \leq (k,\ell) \leq (p,q)\}$ and whose edges are all unit-length north-steps and east-steps connecting its vertices. Every edge $e$ of $G_{p,q}(\bsy{U})$ is assigned a weight $\wt(e)$ given by
\[ \wt(e) = \begin{cases} 
      u_{k,\ell} & \textnormal{if } e \textnormal{ is an east step from } (k,\ell) \textnormal{ to } (k+1,\ell), \\
      v_{k,\ell} & \textnormal{if } e \textnormal{ is a north step from } (k,\ell) \textnormal{ to } (k,\ell+1),
   \end{cases}
\]
as shown in \cref{fig:U}(a):

\begin{figure}[H]
\begin{center}
\begin{tikzpicture}[scale=0.6]

\begin{scope}[xshift=0cm]
\node at (-2,8) {\large (a)};

\draw[step=2cm, thick] (0,0) grid (6,8);

\foreach \j in {0,1,2,3}
{
\foreach \i in {0,1,2}
{
\node at (.45+2*\i,1+2*\j) {\tiny $v_{\i,\j}$};
\node at (1+2*\i,.3+2*\j) {\tiny $u_{\i,\j}$};
\draw[->, thick] (2*\i,2*\j) -- (2*\i+1.7,2*\j);
\draw[->, thick] (2*\i,2*\j) -- (2*\i,2*\j+1.7);
}
}
\foreach \i in {0,1,2}
{
\node at (1+2*\i,.3+2*5-2) {\tiny $u_{\i,4}$};
\draw[->, thick] (2*\i,2*5-2) -- (2*\i+1.7,2*5-2);
}
\foreach \j in {0,1,2,3}
{
\node at (.45+2*4-2,1+2*\j) {\tiny $v_{3,\j}$};
\draw[->, thick] (2*4-2,2*\j) -- (2*4-2,2*\j+1.7);
}
\end{scope}

\begin{scope}[xshift=11cm]
\node at (-2,8) {\large (b)};
\draw[step=2cm, dotted, thin] (0,0) grid (6,8);
   
\foreach \j in {1,2,3,4}
{
\foreach \i in {1,2,3}
{
\node at (.3+2*\i-2,1+2*\j-2) {\tiny \i};
\node at (1+2*\i-2,.3+2*\j-2) {\tiny \j};
}
}

\foreach \i in {1,2,3}
{
\node at (1+2*\i-2,.3+2*5-2) {\tiny 5};
}
\foreach \j in {1,2,3,4}
{
\node at (.3+2*4-2,1+2*\j-2) {\tiny 4};
}

\draw[black,line width=.8mm,dashed] (0,0)--(0,4)--(4,4)--(4,6)--(6,6)--(6,8);
\end{scope}
\end{tikzpicture}
\end{center}
\caption{(a) For $\bsy{U}=\{z_{k,\ell}=(u_{k,\ell},v_{k,\ell}):(0,0)\leq(k,\ell)\leq(3,4)\}$, we show the edge weights on the directed graph $G_{3,4}(\bsy{U})$. Subsequent figures will omit the arrows. (b) We show the path $P=NNEENEN\in\mathcal{L}(3,4)$ and the weights on the edges of $G_{3,4}(\bsy{U})$ for $\bsy{U}$ from \cref{ex:two-dim}. The edge weights on $P$ are 
$(\wt(e_1),\ldots,\wt(e_7))=(1,1,3,3,3,4,4)$, corresponding to the weight sequences $(3,3,4)$ and $(1,1,3,4)$ when restricted to the horizontal and vertical edges, respectively.}\label{fig:U}
\end{figure}

For a lattice path $P\in \mathcal{L}(p,q)$, we write $P = e_1 e_2 \ldots e_{p+q}$, where $e_i \in \{E,N\}$, to record the sequence of steps of $P$. Let $(\bsy{a},\bsy{b})$ be a pair of nonnegative integer sequences with $\bsy{a} = (a_0, \ldots, a_{p-1})$ and $\bsy{b} = (b_0, \ldots, b_{q-1})$ with respective order statistics $a_{(0)} \leq \cdots \leq a_{(p-1)}$ and $b_{(0)} \leq  \cdots \leq b_{(q-1)}$. Then we say that the pair $(\bsy{a},\bsy{b})$ is \emph{bounded by $P$ with respect to $\bsy{U}$} if, for $r = 1, \ldots, p+q$,
\[ \begin{cases} 
      a_{(i)} < \wt(e_r) & \textnormal{if } e_r \textnormal{ is the $i$-th east step of } P, \text{ or}\\
      b_{(j)} < \wt(e_r)  & \textnormal{if } e_r \textnormal{ is the $j$-th north step of } P.
   \end{cases}
\]
It should be noted that the bound is \emph{strict}, similar to the concept of the right boundary. See \cref{ex:two-dim} for an example.

We now provide the definition of two-dimensional vector parking functions as established in \cite{KLY14}.
\begin{definition}[{\cite[Definition 3]{KLY14}}] \label{def:2Dvectorparkingfunctions}
Suppose $\bsy{U} =  \{z_{k,\ell}=(u_{k,\ell},v_{k,\ell}) : (0,0) \leq (k,\ell) \leq (p,q) \} \subset \mathbb{N}^2$ is a set of nodes satisfying $u_{k,\ell} \leq u_{k',\ell'}$ and $v_{k,\ell} \leq v_{k',\ell'}$ when $k \leq k'$ and $\ell \leq \ell'$. A pair of sequences $(\bsy{a},\bsy{b})\in\mathbb{N}^p\times\mathbb{N}^q$ is a \emph{two-dimensional $\bsy{U}$-parking function} if and only if $(\bsy{a},\bsy{b})$ are bounded by some lattice path $P\in\mathcal{L}(p,q)$ with respect to $\bsy{U}$. 
We refer to $\bsy{U}$ as the \emph{weight matrix}. 
We denote the sets of two-dimensional $\bsy{U}$-parking functions by $\PF^{(2)}(\bsy{U})$
and the increasing ones by $\IPF^{(2)}(\bsy{U})$. 
To highlight the dimensions of $\bsy{U}$, we may also use the  notation $\PF^{(2)}_{p,q}(\bsy{U})$ and $\IPF^{(2)}_{p,q}(\bsy{U})$.
\end{definition}

Note that the lattice path, if it exists, is not necessarily unique.

\begin{example}\label{ex:two-dim}
Let $(p,q)=(3,4)$, and 
$\bsy{U}= \{z_{k,\ell}=(\ell+1,k+1) : (0,0) \leq (k,\ell) \leq (3,4) \}$. Consider the pair $\bsy{a}=(2,3,2)$ and $\bsy{b}=(3,0,1,0)$. The path $P=NNEENEN$, shown in \cref{fig:U}(b), has horizontal and vertical edge weights $(3,3,4)$ and $(1,1,3,4)$, which bound the order statistics of $\bsy{a}$ and $\bsy{b}$, respectively. Thus $(\bsy{a},\bsy{b})$ is a $\bsy{U}$-parking function. 

On the other hand, consider $\bsy{a'}=(4,3,3)$ and $\bsy{b'}=(2,0,2,1)$. Suppose there exists a lattice path $P'\in\mathcal{L}(3,4)$ whose horizontal steps have weights strictly greater than $(3,3,4)$ with respect to $\bsy{U}$. 
Then $P'$ must begin with three $N$ steps, making the weights on its first three vertical steps $1$, which means the horizontal weights of $P'$ cannot be strictly greater than $(0,1,2,2)$. Thus $(\bsy{a'},\bsy{b'})$ is not a $\bsy{U}$-parking function as there is no lattice path bounding $(\bsy{a'},\bsy{b'})$ with respect to $\bsy{U}$.
\end{example}

In \cite{SY22}, Snider and Yan showed that $(p,q)$-parking functions are in fact two-dimensional $\bsy{U}$-parking functions. We state their result as follows.
\begin{theorem}[{\cite[Theorem 3.1]{SY22}}]\label{thm:pq is U}
Set $\bsy{U_0}=\{z_{k,\ell}=(\ell+1,k+1) \, : \, (0,0)\leq (k,\ell) \leq (p,q)\}$. A pair 
 $(\bsy{a}, \bsy{b}) \in \mathbb{N}^p \times \mathbb{N}^q$ is a $(p,q)$-parking function if and only if $(\bsy{a}, \bsy{b})\in\PF^{(2)}_{p,q}(\bsy{U_0})$.
\end{theorem}

In general, the exact value of the number of $\bsy{U}$-parking functions is hard to compute. However, when $\bsy{U}$ is given by an affine linear transformation, an explicit formula is known.

\begin{theorem}[{\cite[Corollary 4.2]{LY16}}]
    Fix $p,q \in \mathbb{N}$. Let $\bsy{U}=\{z_{k,\ell}=(u_{k,\ell}, v_{k,\ell}): (0,0) \leq (k,\ell) \leq (p,q)\}$  be given by the affine linear function 
\begin{equation*}
\left(\begin{array}{c} 
u_{k,\ell} \\ 
v_{k,\ell} 
\end{array} \right) =\left( \begin{array}{cc} 
  a & b  \\
  c & d 
  \end{array} 
  \right)
  \left(\begin{array}{c} 
k \\ 
\ell
\end{array} \right) + 
\left(\begin{array}{c} 
s \\ 
t
\end{array} \right), 
\end{equation*} 
  with $a, b, c, d, s, t \in \mathbb{N}$.  
Then the number of increasing $\bsy{U}$-parking functions of size $(p,q)$ is  
\begin{equation} \label{eq:in-goncarov}
\#\IPF^{(2)}_{p,q}(\bsy{U}) = \frac{1}{p!q!}(st+tbq+scp)(s+ap+bq+1)^{(p-1)} (t+cp+dq+1)^{(q-1)} \, ,
\end{equation}
where $x^{(n)}=x(x+1)\cdots (x+k-1) = k! \binom{x+k-1}{k}$.  
The number of $\bsy{U}$-parking functions is 
\begin{equation}\label{eq:goncarov}
\#\PF^{(2)}_{p,q}(\bsy{U})=(st+tbq+scp) (s+ap+bq)^{p-1} (t+cp+dq)^{q-1}. 
\end{equation}
\end{theorem}

\begin{remark} 
We explain  \cref{eq:in-goncarov,eq:goncarov} for the case $pq=0$. When $p=0$ and $q>0$, we have $(\bsy{a}, \bsy{b})=(\bsy{a}, \emptyset)$, where $\bsy{a}$ is a $\bsy{u_1}$-vector parking function associated to the vector $\bsy{u_1}=(u_0,\ldots,u_{q-1})$ with $u_i=s+ai$, and the formulas reduce to those in \cref{Theorem:classicalUPF}, with the convention that $x^{(-1)}:=1/(x-1)$. The $p>0$ and $q=0$ case is symmetric.  Finally, when $p=q=0$, both formulas equal $1$, counting the unique parking function $(\bsy{a},\bsy{b})=(\emptyset, \emptyset)$. 
\end{remark}

\section{Prime vector parking functions} \label{section:3} 

Recall that one of the characterizations of a prime classical parking function is the requirement that the weak order statistic inequalities are made strict, as per \cref{def:classical prime}. We extend this notion to define prime $\bsy{u}$-parking functions, by making the inequalities in Ineq. \eqref{eq:u inequalities} strict.

\begin{definition} \label{def:prime vector pf}
Let  $\bsy{u}=(u_0,u_1,\ldots,u_{n-1})$ be a weakly increasing vector of nonnegative integers.
A $\bsy{u}$-parking function is \emph{prime} if 
\begin{align}\label{eq:prime upf def}
&\#\{j: \ a_j <  u_i\} > i+1 \; \text{ for } \; i =0,1,\dots,n-2\,.
\end{align}
In particular, if $n=1$, any $\bsy{u}$-parking function is prime.  
\end{definition}

Denote the set of prime $\bsy{u}$-parking functions by $\PPF_n(\bsy{u})$ and the set of increasing ones by $\IPPF_n(\bsy{u})$.  
Analogous to \cref{lem:classical prime equivalent}, 
there is an equivalent description for prime vector parking functions.  

\begin{theorem} 
Let  $n \geq 2$, $\bsy{u}=(u_0, u_1, \dots, u_{n-1})$, and $\bsy{u_1}=(u_0, u_1, \dots, u_{n-2})$. For $\bsy{a}\in\mathbb{N}^n$, $\bsy{a}\in \PPF_n(\bsy{u})$ if and only if removing any entry less than $u_0$ from $\bsy{a}$ yields a $\bsy{u_1}$-parking function of length $n-1$.
\end{theorem}
\begin{proof}
    Suppose $\bsy{a}\in\PPF_n(\bsy{u})$. Since  
    $\#\{j: \ a_j < u_0\} > 1$, there is an entry of $\bsy{a}$ less than $u_0$. Let $\bsy{a'}=(a_0',\ldots,a_{n-1}')$ be any sequence obtained by removing such an entry. For $i=0,\dots,n-2$, we have $\#\{j:a_j'<u_i\}=\#\{j:a_j<u_i\}-1\geq i+1$, so $\bsy{a'}\in PF_{n-1}(\bsy{u_1})$.
    Reversing the argument gives the other direction. 
\end{proof}

Comparing \cref{vectorpark}  and \cref{def:prime vector pf}, we get another characterization of prime $\bsy{u}$-parking functions.
\begin{proposition}\label{prop:u prime is u'} 
Let $n \geq 1$, $\bsy{u}=(u_0,u_1,\ldots,u_{n-1})$, and $\bsy{u}'=(u_0, u_0, u_1,\dots, u_{n-2})$. Then a vector $\bsy{a}$ is a prime $\bsy{u}$-parking function if and only if it is a $\bsy{u}'$-parking function.
\end{proposition}

\begin{proof}
It is obvious for $n=1$. For $n \geq 2$, 
rewriting Ineq. \eqref{eq:prime upf def} as $\#\{j: \ a_j <  u_{i}\} \geq i+2$ for all $i \in \{0,1,\dots,n-2\}$,
we reindex to obtain $\#\{j: \ a_j <  u_{i-1}\} \geq i+1$ for all $i\in \{1,\ldots,n-1\}$, which coincides with the set of inequalities defining a $\bsy{u}'$-parking function.
\end{proof}

From the above lemma, there is an immediate equivalence between prime $\bsy{u}$-parking functions and $\bsy{u'}$-parking functions, as well as for their increasing counterparts:
\begin{align*}
\PPF_n(\bsy{u})&\longleftrightarrow \PF_n(\bsy{u'})\\
\IPPF_n(\bsy{u})&\longleftrightarrow \IPF_n(\bsy{u'})
\end{align*}
where 
\[
\bsy{u}=(u_0, u_1, \ldots,u_{n-1}),  \qquad   \bsy{u'}=(u_0,u_0,u_1, \ldots,u_{n-2}).
\]

\subsection{Decomposition into prime components} 
We justify the definition of prime vector parking functions by giving an explicit description of how a 
$\bsy{u}$-parking function decomposes into prime components. 

Recall that increasing $\bsy{u}$-vector parking functions can be represented by lattice paths with right boundary $\bsy{u}$, and $\bsy{u}$-vector parking functions are obtained by labeling the vertical edges of the lattice path. We first examine the lattice paths that correspond to  prime $\bsy{u}$-vector parking functions; we refer to them as \emph{$\bsy{u}$-prime lattice paths}, or lattice paths that are \emph{prime with respect to right boundary $\bsy{u}$}. Define the set of points $\mathcal{I}_{\bsy{u}}=\{(0,0),(u_0,1),(u_1,2),\ldots,(u_{n-1},n)\}$. The $\bsy{u}$-prime lattice paths in $\mathcal{L}(u_{n-1}, n)$ are precisely those paths that contain points in $\mathcal{I}_{\bsy{u}}$ only at the beginning and ending vertices (namely, $\mathcal{I}_{\bsy{u}}\cap P= \{(0,0), (u_{n-1},n)\}$). In particular, the $\bsy{u}$-prime components of a given lattice path $P\in\Lc(u_{n-1},n)$ with the right boundary $\bsy{u}$ are associated to the portions of $P$ that lie between the points in the set $\mathcal{I}_{\bsy{u}}\cap P$.

Assume $\mathcal{I}_{\bsy{u}}\cap P=\{(0,0),(u_{i_1},i_1+1),\ldots,(u_{i_k},i_k+1)\}$ where $0< i_1< i_2< \dots < i_k=u_{n-1}$ for some $k\geq 1$ is the set of points in $\mathcal{I}_{\bsy{u}}$ on the path $P$. Let $P_1$ be the portion of $P$ from $(0,0)$ to $(u_{i_1}, i_1+1)$ and $P_j$ be the portion of $P$ from $(u_{i_{j-1}}, i_{j-1}+1)$ to $(u_{i_{j}}, i_{j}+1)$ for $j=2, \dots, k$. Define $\bsy{u^{(1)}}=(u_0, u_1,\dots,u_{i_1})$, and for $2\leq j\leq k$, define
  \[
\bsy{u^{(j)}} = ( u_{i_{j-1}+1}- u_{i_{j-1}}, u_{i_{j-1}+2}-u_{i_{j-1}}, \dots, u_{i_j}-u_{i_{j-1}}). 
\]
Then $P=P_1\oplus P_2 \oplus \cdots\oplus P_{k}$ is a concatenation of lattice paths such that $P_j$ is a $\bsy{u^{(j)}}$-prime lattice path for $1\leq j\leq k$. Indeed, $P_1\in\Lc(u_{i_1}, i_1+1)$ is prime with respect to the right boundary $\bsy{u^{(1)}}$, and, for  $j=2, \dots, k$, viewing the point $(u_{i_{j-1}}, i_{j-1}+1)$ as the origin $(0,0)$, we confirm that 
$P_j\in\Lc(u_{i_j}-u_{i_{j-1}}, i_j-i_{j-1})$ is prime with respect to the right boundary
$\bsy{u^{(j)}}$.

Now let $\bsy{i}$ be an increasing $\bsy{u}$-parking function with corresponding lattice path (with right boundary  $\bsy{u}$) $P$. Then $\bsy{i}$ can be decomposed as a direct sum of prime increasing vector parking functions  $\bsy{i} = \bsy{i^{(1)}} \oplus \bsy{i^{(2)}} \oplus \cdots \oplus \bsy{i^{(k)}}$ where $\bsy{i^{(j)}}\in\IPPF(\bsy{u^{(j)}})$ is the weakly increasing sequence corresponding to the lattice path $P_j$, for $1\leq j\leq k$. See \cref{ex:u decompose}. 

\begin{example}\label{ex:u decompose}
    Let $\bsy{u}=(1,2, 4, 5, 7,8)$ with $L_{\bsy{u}}$ drawn in blue, and let $\bsy{i}=(0,0, 3, 3, 4, 7) \in \IPF(\bsy{u})$ with $L_{\bsy{i}}$ drawn in red. We mark the points in $\mathcal{I}_{\bsy{u}}$ with (solid or empty) circles to show the prime decomposition of $\bsy{i}$: The prime components correspond to portions of the path separated by the (solid blue) points in the set $\mathcal{I}_{\bsy{u}}\cap L_{\bsy{i}}=\{(0,0),(u_1,2),(u_4,5),(u_5,6)\}$.   For this example, ignore the sequence $\bsy{a}$ and the labels on the lattice paths.

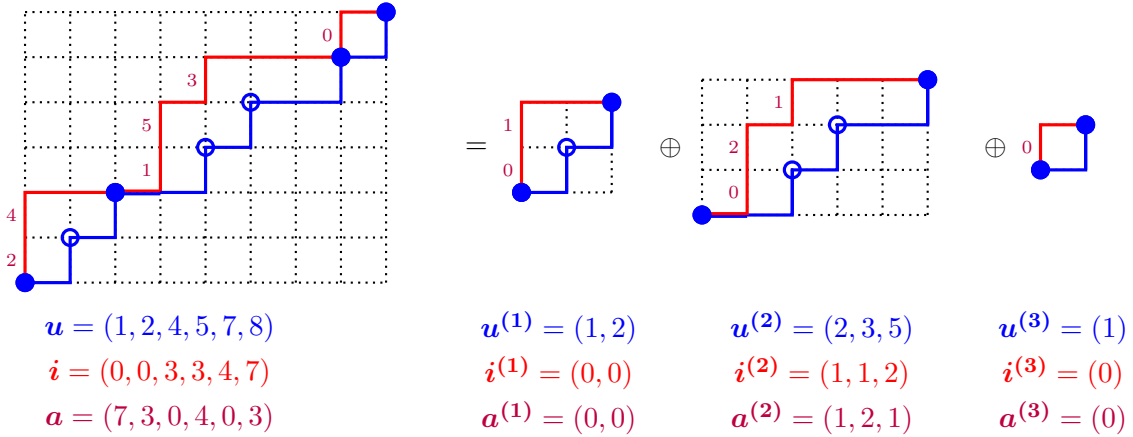
\begin{center}
\begin{tikzpicture}[scale=0.6]

\draw[step=1cm, thick, dotted] (0,0) grid (8,6);
\foreach \i\u in {0/0,1/1,2/2,3/4,4/5,5/7,6/8}
{
\draw[blue,very thick] (\u,\i) circle (5pt);
}

\draw[red,very thick]  (0,0)--(0,2)--(2,2);
\draw[red,very thick, yshift=0.03cm] (2,2)--(3,2);
\draw[red,very thick] (3,2)--(3,4)--(4,4)--(4,5)--(7,5)--(7,6)--(8,6);

\node[purple] at (-.3, .5) {\tiny{2}}; 
\node[purple] at (-.3, 1.5) {\tiny {4}}; 
\node[purple] at (2.7, 2.5) {\tiny {1}}; 
\node[purple] at (2.7, 3.5) {\tiny {5}}; 
\node[purple] at (3.7, 4.5) {\tiny {3}}; 
\node[purple] at (6.7, 5.5) {\tiny {0}}; 

\draw[blue, very thick] (0,0)--(1,0)--(1,1)--(2,1)--(2,2);
\draw[blue, very thick, yshift=-0.03cm] (2,2)--(3,2);
\draw[blue, very thick] (3,2)--(4,2)--(4,3)--(5,3)--(5,4)--(7,4)--(7,5)--(8,5)--(8,6);

\foreach \i\u in {0/0,2/2,5/7,6/8}
{
\filldraw[blue,very thick] (\u,\i) circle (5pt);
}

\node at (10,3) {$=$};
\node at (14.3,3) {$\oplus$};
\node at (21.5,3) {$\oplus$};


\begin{scope}[xshift=11cm,yshift=2cm]
\draw[step=1cm, thick, dotted] (0,0) grid (2,2);

\foreach \i\u in {0/0,1/1,2/2}
{
\draw[blue,very thick] (\u,\i) circle (5pt);
}

\draw[blue, very thick] (0,0)--(1,0)--(1,1)--(2,1)--(2,2); 
\draw[blue, very thick] (1,0)--(1,1) (2,1)--(2,2); 
\draw[red, very thick] (0,0)--(0,2)--(2,2);
\foreach \i\u in {0/0,2/2}
{
\filldraw[blue,very thick] (\u,\i) circle (5pt);
}

\node[purple] at (-.3, .5) {\tiny{0}}; 
\node[purple] at (-.3, 1.5) {\tiny{1}}; 
\end{scope}


\begin{scope}[xshift=15cm,yshift=1.5cm]
\draw[step=1cm, thick, dotted] (0,0) grid (5,3);

\foreach \i\u in {0/0,1/2,2/3,3/5}
{
\draw[blue,very thick] (\u,\i) circle (5pt);
}

\draw[blue, very thick, yshift=-0.03cm] (0,0)--(1,0);
\draw[blue, very thick] (1,0)--(2,0)--(2,1)--(3,1)--(3,2)--(5,2)--(5,3);
\draw[blue, very thick] (2,0)--(2,1) (3,1)--(3,2) (5,2)--(5,3);
\draw[red, very thick, yshift=0.03cm] (0,0)--(1,0);
\draw[red, very thick] (1,0)--(1,2)--(2,2)--(2,3)--(5,3);
\foreach \i\u in {0/0,3/5}
{
\filldraw[blue,very thick] (\u,\i) circle (5pt);
}

\node[purple] at (0.7, .5) {\tiny{0}}; 
\node[purple] at (0.7, 1.5) {\tiny{2}}; 
\node[purple] at (1.7, 2.5) {\tiny{1}}; 
\end{scope}


\begin{scope}[xshift=22.5cm,yshift=2.5cm]
\draw[step=1cm, thick, dotted] (0,0) grid (1,1);

\foreach \i\u in {0/0,1/1}
{
\draw[blue,very thick] (\u,\i) circle (5pt);
}

\draw[red, very thick] (0,0)--(0,1)--(1,1);
\draw[blue, very thick] (0,0)--(1,0)--(1,1);
\draw[blue, very thick] (1,0)--(1,1);
\foreach \i\u in {0/0,1/1}
{
\filldraw[blue,very thick] (\u,\i) circle (5pt);
}

\node[purple] at (-.3, .5) {\tiny{0}}; 
\end{scope}

\begin{scope}[yshift=1cm]
\node[blue] at (3,-2) {$\bsy{u}=(1,2, 4,5, 7, 8)$};
\node[red] at (3,-3) {$\bsy{i}=(0,0,3,3,4,7)$};
\node[purple] at (3, -4) {$\bsy{a}=(7,3,0,4,0,3)$}; 

\node[blue,xshift=6.5cm] at (1,-2) {$\bsy{u^{(1)}}=(1,2)$};
\node[red,xshift=6.5cm] at (1,-3) {$\bsy{i^{(1)}}=(0,0)$};
\node[purple,xshift=6.5cm] at (1,-4) {$\bsy{a^{(1)}}=(0,0)$};

\node[blue,xshift=10cm] at (1,-2) {$\bsy{u^{(2)}}=(2,3,5)$};
\node[red,xshift=10cm] at (1,-3) {$\bsy{i^{(2)}}=(1,1,2)$};
\node[purple,xshift=10cm] at (1,-4) {$\bsy{a^{(2)}}=(1,2,1)$};

\node[blue,xshift=13.5cm] at (.5,-2) {$\bsy{u^{(3)}}=(1)$};
\node[red,xshift=13.5cm] at (.5,-3) {$\bsy{i^{(3)}}=(0)$};
\node[purple,xshift=13.5cm] at (.5,-4) {$\bsy{a^{(3)}}=(0)$};
\end{scope}

\end{tikzpicture}
\end{center} 
\end{example}

General $\bsy{u}$-parking functions are obtained by labeling the vertical edges of a lattice path $P \in \Lc(u_{n-1}, n)$ with the right boundary $\bsy{u}$, with the property that labels along each of the vertical edges with the same $x$-coordinates increase from bottom to top. 
Similar to the classical case, any $\bsy{u}$-vector parking function can be viewed as a direct (shuffle) sum of prime vector parking functions  by the following steps. 

Let $\bsy{a}$ be a $\bsy{u}$-parking function with corresponding labeled lattice path  $P$. Ignoring the labeling, decompose $P$ as the concatenation $P_1 \oplus P_2 \oplus \cdots \oplus P_k$ as described above. For $1\leq i\leq k$, let $B_i$ be the set of labels on $P_i$.  Replace the labels in $B_i$ by $\{0, 1, \dots, |B_i|-1\}$, following  the same numerical order, to obtain a lattice path $P_i'$ whose vertical edges are labeled by {$\mathbb{N}_{|B_i|}$}  and let $\bsy{a^{(i)}}$ be the $\bsy{u^{(i)}}$-vector parking function corresponding to the labeled lattice path $P_i'$. 
Then we write 
\[
\bsy{a}= \bsy{a^{(1)}} \oplus \bsy{a^{(2)}} \oplus  \cdots \oplus \bsy{a^{(k)}}. 
\]
Another way to obtain $\bsy{a^{(i)}}$ is to let $\bsy{b^{(i)}}$ be the subsequence of $\bsy{a}$ restricted to the positions in $B_i$ and take $\bsy{a^{(i)}}=\bsy{b^{(i)}}-u_{p_i-1}$, where $p_i=\sum_{j < i} |B_j|$.

\begin{example}\label{ex:u PF decompose}
    For $\bsy{u}=(1,2,4,5,7,8)$ and $\bsy{a}=(7,3,0,4,0,3)\in\PF(\bsy{u})$, 
    the increasing rearrangement of $\bsy{a}$ is $\bsy{i}$ in \cref{ex:u decompose}.
    The labels on the vertical edges of $L_{\bsy{a}} = L_{\bsy{i}}$ are $(2,4,1,5,3,0)$ from bottom to top, so that $B_1=\{2,4\}$, $B_2=\{1,3, 5\}$ and $B_3=\{0\}$. Therefore $\bsy{b^{(1)}}=(0,0)$, 
    $\bsy{b^{(2)}}=(3,4,3)$ and $\bsy{b^{(3)}}=(7)$. It follows that  $\bsy{a^{(1)}}=\bsy{b^{(1)}}=  (0,0)$,
    $\bsy{a^{(2)}} = \bsy{b^{(2)}}-u_1=(1,2,1)$ and $\bsy{a^{(3)}}= \bsy{b^{(3)}}-u_{4}=(0) $. 
    The labeling on the lattice path $L_{\bsy{a}}$ is indicated in the figure in \cref{ex:u decompose}. 

\end{example}

\subsection{Enumeration of prime $\bsy{u}$-parking functions when $\bsy{u}$ is an arithmetic progression} \label{subsection:enumeration upf}

In this subsection we derive explicit formulas for  prime $\bsy{u}$-parking functions when $\bsy{u}=(u_0, \ldots, u_{n-1})$  is an arithmetic sequence given by $u_i=a+bi$. 
First we consider  $\IPPF_n(\bsy{u})$, the set of increasing prime $\bsy{u}$-parking functions.  
\begin{theorem}\label{Theorem:IncreasingPrimeAP}
    Fix $n\geq 1$ and let $a, b \in \mathbb{N}$. 
    The number of increasing prime $\bsy{u}$-parking functions for $\bsy{u}=(u_0,\ldots,u_{n-1})$ given by $u_i=a+bi$ is
    \begin{equation}\label{eq:IPPU}
    \#\IPPF_n(\bsy{u})=\frac{a-b}{n}\binom{a+(b+1)(n-1)}{n-1} + \frac{b}{n}\binom{(b+1)(n-1)}{n-1}.
    \end{equation}
\end{theorem}
When $a=b=1$, \cref{Theorem:IncreasingPrimeAP} recovers the enumeration for the classical case $\IPPF(\bsy{u})=\IPPF(n)=C_{n-1}$. 
\begin{proof}
Recall that $\bsy{\alpha}$ is an increasing prime $\bsy{u}$-parking function if and only if $\bsy{\alpha}$ is an increasing $\bsy{u}'$-parking function for $\bsy{u}'=(a,a,a+b,a+2b,\ldots,a+(n-2)b)$. For each $i\in\{0,1,\ldots,a-1\}$, by \cref{Theorem:classicalUPF}, the number of $\bsy{\alpha}=(a_0,a_1,\ldots,a_{n-1})\in \IPPF_n(\bsy{u})$ with $a_0=i$ is
\begin{eqnarray*}
m_{n,i}&:= &\frac{a-i}{a+(n-1)(b+1)-i}\binom{a+(n-1)(b+1)-i}{n-1} \\
  & = &\left( 1-\frac{K}{K+a-i}\right) \binom{K+a-i}{n-1}     \\
  & =  &  \binom{K+a-i}{n-1} - \frac{K}{n-1} \binom{K+a-i-1}{n-2} 
\end{eqnarray*}
where $K = (n-1)(b+1)$.
Next, we have
\begin{align*}
    &\sum_{i=0}^{a-1}\binom{K+a-i}{n-1} - \frac{K}{n-1} \binom{K+a-i-1}{n-2}\\
    &=\sum_{i=K+1}^{K+a} \binom{i}{n-1}-\frac{K}{n-1}\sum_{i=K}^{K+a-1}\binom{i}{n-2} \quad \text{by reindexing}\\
    &=\left(\sum_{i=0}^{K+a}\binom{i}{n-1}-\sum_{i=0}^{K}\binom{i}{n-1}\right)-\frac{K}{n-1}\left(\sum_{i=0}^{K+a-1}\binom{i}{n-2}-\sum_{i=0}^{K-1}\binom{i}{n-2}\right)\\
    &=\binom{K+a+1}{n}-\binom{K+1}{n}-\frac{K}{n-1}\left(\binom{K+a}{n-1}-\binom{K}{n-1}\right)\quad \text{by the Hockey-stick identity}\\
    &=\binom{K+a+1}{n}-\frac{K}{n-1}\binom{K+a}{n-1}+\frac{K}{n-1}\binom{K}{n-1}-\binom{K+1}{n}.
\end{align*}
The above can be simplified using 
\begin{align*}
    \binom{K+a+1}{n}-\frac{K}{n-1}\binom{K+a}{n-1}
    =\frac{a-b}{n}\binom{K+a}{n-1}
    \intertext{and}
    \frac{K}{n-1}\binom{K}{n-1}-\binom{K+1}{n}
   =\frac{b}{n}\binom{K}{n-1}.
\end{align*}
Combining the above, we get
\begin{align*}
    \sum_{i=0}^{a-1}\binom{K+a-i}{n-1} - \frac{K}{n-1} \binom{K+a-i-1}{n-2}&=\frac{a-b}{n}\binom{K+a}{n-1}+\frac{b}{n}\binom{K}{n-1} 
\end{align*}
as desired. 
\end{proof}

Next we enumerate $\PPF_n(\bsy{u})$, the set of all prime $\bsy{u}$-parking functions when $\bsy{u}$ is given by an  arithmetic progression.  

\begin{theorem}\label{Theorem:primeAP}
    Fix $n\geq 1$ and let $a, b \in \mathbb{N}$. 
    The number of  prime $\bsy{u}$-parking functions for $\bsy{u}=(u_0,\ldots,u_{n-1})$ given by $u_i=a+bi$ is
    \begin{equation}\label{eq:PPF(u)}
    \# \PPF_n(\bsy{u})=(a-b)\big[ a+(n-1)b \big]^{n-1} + b^n (n-1)^{n-1}. 
    \end{equation}
\end{theorem}

\begin{proof}
Following \cref{prop:u prime is u'}, we count the number of $\bsy{u'}$-parking functions where $\bsy{u'}=(a,a,a+b,a+2b,\ldots,a+(n-2)b)$. For a $\bsy{u'}$-parking function $\bsy{a}$, let $\bsy{a}_1$ be the subsequence of $\bsy{a}$ consisting of entries smaller than $a$ and $\bsy{a}_2$ be the subsequence of $\bsy{a}$ consisting of entries larger than or equal to $a$. 

Let $k:=|\bsy{a}_1|$. Then, notice that $\bsy{a}_1 \in \{0, 1, \dots, a-1\}^k$ and 
$\bsy{a}_2 \in \PF(a+(k-1)b, a+kb, \dots, a+(n-2)b)$. 
Since every entry in $\bsy{a}_2$ is at least $a$, we have that $\bsy{a}_2-a$ is a $\bsy{u^{(k)}}$-parking function, where $\bsy{u^{(k)}}=((k-1)b, kb, \dots, (n-2)b)$.  Since $\bsy{u^{(k)}}=(u^{(k)}_0,\ldots,u^{(k)}_{n-k})$ is given by the arithmetic progression $u^{(k)}_i=a'+bi$ with $a'=(k-1)b$, 
the number of choices for $\bsy{a}_2$ is $\#\PF_{n-k}(\bsy{u^{(k)}}) = (k-1)b ( nb-b)^{n-k-1}$. 
Thus the number of choices for $(\bsy{a}_1,\bsy{a}_2)$ such that $\bsy{a}\in\PPF_n(\bsy{u'})$ is $a^k (k-1)b ( nb-b)^{n-k-1}$. 

As $k=|\bsy{a}_1|\geq 1$, summing over $1\leq k\leq n$, we obtain that $\#\PPF(\bsy{u})$ is equal to
\begin{eqnarray*}
  \sum_{k=1}^n  \binom{n}{k} a^k (k-1)b (nb-b)^{n-k-1} 
   & = & \frac{1}{n-1} \left( \sum_{k=1}^n \binom{n}{k} k a^k (nb-b)^{n-k}  -\sum_{k=1}^n \binom{n}{k} a^k (nb-b)^{n-k}  \right) \\ 
   & = & \frac{1}{n-1} \big(na(a+nb-b)^{n-1} - (a+nb-b)^n + (nb-b)^{n}     \big) \\
   & = &  b^n (n-1)^{n-1} + (a-b)(a+nb-b)^{n-1}.     
\end{eqnarray*}
\end{proof} 

When $a=b=1$, \cref{Theorem:primeAP} recovers the enumeration for the classical case $\PPF_n(\bsy{u})=\PPF(n)=(n-1)^{n-1}$.

\begin{remark}
For a vector $\bsy{u}=(u_0, u_1, \dots, u_{n-1})$, define the difference vector of $\bsy{u}$ by
    \[\Delta(\bsy{u}):=(u_0, u_1-u_0, u_2-u_1, \dots, u_{n-1}-u_{n-2}).\] 
    Let $\bsy{u}$ be given by $u_i=a+bi$. For  $\bsy{u'}=(u_0,u_0,u_1,\ldots,u_{n-2})$, we have $\Delta(\bsy{u'})=(a, 0, b, b, \dots, b)$. 
    Then \cref{Theorem:primeAP} could alternatively be proven by applying  \cite[Theorem 3]{Yan2000}  
    to $\Delta(\bsy{u'})$.
\end{remark}

\vanish{ 
If we want \# r-prime ones, the formula is 
\[
\sum_{k=r}^n  \binom{n}{k} a^k (k-1)b (nb-b)^{n-k-1} 
= b^n (n-1)^{n-1} + (a-b)(a+nb-b)^{n-1} - \sum_{k=1}^{r-1}  \binom{n}{k} a^k (k-1)b (nb-b)^{n-k-1}. 
\] 
} 

\section{Prime $(p,q)$-parking functions} \label{section:pq parking functions}

In this section, we propose a definition for \emph{prime $(p,q)$-parking functions}, which will serve as the $(p,q)$-analogue of prime classical  parking functions. Our motivation for this definition is to obtain a unique decomposition of a $(p,q)$-parking function into prime components. 

We always assume that at least one of $p,q$ is nonzero unless otherwise specified, but for technical reasons, we allow the case $pq = 0$ in the following definition.
\begin{definition}
\label{IPPFLemma}
    Let $p,q \ge 1$ and let $(\bsy{a},\bsy{b}) \in \mathbb{N}^p \times \mathbb{N}^q$ be a $(p,q)$-parking function with order statistics $a_{(0)}\leq \cdots\leq a_{(p-1)}$ and $b_{(0)}\leq \cdots\leq b_{(q-1)}$ for $\bsy{a}$ and $\bsy{b}$, respectively. Then $(\bsy{a},\bsy{b})$ is \emph{prime} if it satisfies the conditions
    \begin{itemize}
        \item[i)] $\#\{j:a_j<i\}>b_{(i)}$
for $1\leq i\leq q-1$, and
        \item[ii)] $\#\{j:b_j<i\}>a_{(i)}$
for $1\leq i\leq p-1$.
    \end{itemize}
    
    If $p = 0$, then we say that $(\emptyset,\bsy{b}) \in \mathbb{N}^0 \times \mathbb{N}^q$ is prime if and only if $q = 1$ and $\bsy{b} = (0)$.

    If $q = 0$, then we say that $(\bsy{a},\emptyset) \in \mathbb{N}^p \times \mathbb{N}^0$ is prime if and only if $p = 1$ and $\bsy{a} = (0)$.
\end{definition}
We let $\PPF(p,q)$ denote the set of prime $(p,q)$-parking functions, and let $\IPPF(p,q)$ denote the set of increasing prime $(p,q)$-parking functions. Note that the above inequalities imply that entries in $\bsy{a}$ and $\bsy{b}$ are bounded by $q$ and $p$, respectively. Furthermore, it follows immediately from the definition that $(\bsy{a},\bsy{b})$ is a prime $(p,q)$-parking function if and only if $(\bsy{b},\bsy{a})$ is a prime $(q,p)$-parking function. 

As in the classical case, there are a few equivalent ways of defining primeness for $(p,q)$-parking functions, which we provide in the following proposition.

\begin{proposition}\label{prop:equivalent}
  The following are equivalent.  
  \begin{enumerate}
      \item $(\bsy{a},\bsy{b})$ is a prime  $(p,q)$-parking function. 
      
      \item The lattice paths $L^\perp_{\bsy{a}}$ and $L_{\bsy{b}}$ have exactly two common points: $(0,0)$ and $(p,q)$.  

      \item If $p,q \ge 1$, removing a $0$ entry from both $\bsy{a}$ and $\bsy{b}$ yields a $(p-1, q-1)$-parking function. 

      \end{enumerate}
      \end{proposition}
\begin{proof}
    If $p=0$, then both (1) and (2) occur if and only if $q=1$. In this case, the pair $(\bsy{a},\bsy{b})=(\emptyset,(0))$ is the only prime $(0,1)$-parking function, and $(\bsy{a},\bsy{b})=(\emptyset,(0))$ is also the only pair in $\mathbb{N}^0\times\mathbb{N}^1$ whose corresponding lattice paths $L^{\perp}_{\bsy{a}},L_{\bsy{b}}\in\Lc(0,1)$ share exactly two common points. The case $q=0$ is similar. Thus (1) and (2) are equivalent for $pq=0$.
    
    We now prove the equivalences in the case $p,q \geq 1$. Let $(\bsy{a},\bsy{b}) \in \mathbb{N}^p \times \mathbb{N}^q$. 
    First assume that $(\bsy{a},\bsy{b})$ is a prime $(p,q)$-parking function. We will show that (2) holds. Assume for the sake of contradiction that $L^\perp_{\bsy{a}}$ and $L_{\bsy{b}}$ intersect at some lattice point $(c,r) \neq (0,0)$ with $c < p$ or $r < q$.  By symmetry, it suffices to consider the case where $0<c<p$. If $r = 0$, we must have $b_{(0)} \ge c > 0$. But this then implies that $\#\{ j : b_j < 1 \} = 0 \le a_{(0)}$, contradicting (ii.) of \Cref{IPPFLemma}.
    
Now assume that $r \geq 1$. We must then have $a_{(c-1)} \le r$ and $b_{(r-1)} \le c$. We examine the following two cases for $r$. 
    \begin{enumerate}
    \item[i.] If $r < q$, then we must have $a_{(c)} \ge r$ and $b_{(r)} \ge c$. Thus $\#\{ j : a_j < { r} \} \le c \le b_{(r)}$, contradicting i) of \Cref{IPPFLemma}.
    \item[ii.] If $r = q$, then we must have $a_{(c)} = \cdots = a_{(p-1)} = q$, so that $\# \{ j : b_j < p-1 \} \le q = a_{(p-1)}$, contradicting ii) of \Cref{IPPFLemma}.
    \end{enumerate} 
    Thus (1) $\Rightarrow$ (2). Now assume that $(\bsy{a},\bsy{b}) \in \mathbb{N}^p \times \mathbb{N}^q$ is a non-prime, $(p,q)$-parking function. Then without loss of generality, there exists some $0 \le i < q-1$ such that i) of \cref{IPPFLemma} doesn't hold, so that we have $\#\{j:a_j<i\}\leq b_{(i)}$. On the other hand, since $(\bsy{a},\bsy{b})$ is a parking function, we also have $b_{(i)}\leq \#\{j:a_j<i+1\}$. Thus $a_{(b_{(i)})}\geq i+1$, while $a_{(b_{(i)}-1)}<i+1$, so $L^\perp_{\bsy{a}}$ necessarily passes through the point $(b_{(i)},i+1)$. But $L_{\bsy{b}}$ also passes through $(b_{(i)},i+1)$ by construction, so (2) is not satisfied. Therefore (2) $\Rightarrow$ (1), proving the first equivalence.  

    We now prove that conditions (1) and (3) are equivalent. First note that if $(\bsy{a},\bsy{b})$ is a prime $(p,q)$-parking function, then it follows immediately from \Cref{IPPFLemma} that each of $\bsy{a}$ and $\bsy{b}$ contains a $0$ entry. Let $(\bsy{a}',\bsy{b}') \in \mathbb{N}^{p-1} \times \mathbb{N}^{q-1}$ be any pair 
    obtained by removing a $0$ from both $\bsy{a}$ and $\bsy{b}$. Then $a_{(i)}' = a_{(i+1)}$ for each $0 \le i \le p-2$, and $b_{(i)}' = b_{(i+1)}$ for each $0 \le i \le q-2$. Furthermore, note that
    \[
    \# \{ j : a_j' < i+1 \} = \# \{ j : a_j < i+1 \} - 1 \quad \text{and} \quad \# \{ j : b_j' < i+1 \} = \# \{ j : b_j < i+1 \} - 1
    \]
    for any $i \ge 0$. We then have that
\begin{align}
    \#\{j:a_j<i+1\}-1>b_{(i+1)}-1,&\ \mbox{which implies}\ \#\{j:a'_j<i+1\}\geq b_{(i+1)}=b'_{(i)}, \ \mbox{and}\nonumber\\
    \#\{j:b_j<i+1\}-1>a_{(i+1)}-1,&\ \mbox{which implies}\  \#\{j:b'_j<i+1\}\geq a_{(i+1)}=a'_{(i)},
    \label{eq:pq to U}
\end{align}
    so that $(\bsy{a}',\bsy{b}')$ is a $(p-1,q-1)$-parking function. Reversing the argument also proves that if $(\bsy{a}',\bsy{b}')$ is a $(p-1,q-1)$-parking function and $(\bsy{a},\bsy{b})$ is obtained by prepending a $0$ entry to each of $\bsy{a}'$ and $\bsy{b}'$, then the resulting pair $(\bsy{a},\bsy{b})$ is a prime $(p,q)$-parking function. Thus (1) and (3) are equivalent.
\end{proof}

From \cref{prop:equivalent}, by comparing Ineq.~\eqref{eq:pq to U} to \cref{def:2Dvectorparkingfunctions}, we obtain the following useful equivalence, illustrated in \cref{ex:pq prime}.
\begin{proposition}
    \label{prop:2vector}
Let $\bsy{U_0'}=\{ (u_{k,\ell},v_{k,\ell}):(0,0)\leq (k,\ell)\leq (p,q)\}$ be defined by  
\[
          (u'_{k,\ell},v'_{k,\ell}) = \begin{cases} 
            (\ell, k) &  \text{ if } k, \ell\geq 1 \\ 
            (1,1)  & \text{ if } k\ell=0.
            \end{cases} 
\]
The prime $(p,q)$-parking functions are exactly the two-dimensional $\bsy{U_0'}$-parking functions, and the same for their increasing counterparts. That is,
\[
\PPF(p,q) = \PF_{p,q}^{(2)}(\bsy{U_0'})\qquad \text{and} \qquad
\IPPF(p,q) = \IPF_{p,q}^{(2)}(\bsy{U_0'})\, . 
\]

  \end{proposition}
  \begin{proof} 
      Let $(\bsy{a'},\bsy{b'})\in\PF(p-1,q-1)$ be the $(p-1,q-1)$ parking function obtained from $(\bsy{a},\bsy{b})$ by removing a 0 from each vector according to \cref{prop:equivalent}, so that $a'_i=a_{i+1}$ for $0\leq i\leq p-2$ and $b'_j=b_{j+1}$ for $0\leq j\leq q-2$. By \cref{thm:pq is U},  we have that $(\bsy{a'},\bsy{b'})\in\PF^{(2)}_{p-1,q-1}(\bsy{U_0})$, where $\bsy{U_0}=\{ (u_{k,\ell},v_{k,\ell}):(0,0)\leq (k,\ell)\leq (p-1,q-1)\}$ is given by 
      \[(u_{k,\ell},v_{k,\ell})=(k+1,\ell+1)=\begin{cases}
          (u'_{k,\ell}+1,v'_{k,\ell}+1)&k,\ell\geq 1\\(u'_{k,\ell},v'_{k,\ell})&k\ell=0.
      \end{cases}\] 
      Let $P'\in\Lc(p-1,q-1)$ be a path which bounds $(\bsy{a'},\bsy{b'})$ with respect to $\bsy{U_0}$, and consider the path $P=\{EN\}\oplus P'\in\Lc(p,q)$. The weights of all but the first two edges of $P$ lie off the $x=0$ and $y=0$ axes, and as the $x$- and $y$-coordinates of those edges are incremented by 1 while $(u'_{k,\ell},v'_{k,\ell})=(u_{k,\ell}-1,v_{k,\ell}-1)$ for all $k,\ell\geq 1$, the weights of those edges with respect to $\bsy{U_0'}$ are identical to their weights with respect to $\bsy{U_0}$.  Thus the last $p-1$ horizontal edges of $P$ bound those of $\bsy{a}$ and the last $q-1$ vertical edges of $P$ bound those of $\bsy{b}$. Moreover, the first edge of $\bsy{a}$ (resp.~$\bsy{b}$) is $a_0=0$ (resp.~$b_0=0$), which is bounded by $u'_{0,1}=1$ (resp.~$v'_{1,0}=1$). Thus $P$ bounds $(\bsy{a},\bsy{b})$ with respect to $\bsy{U_0'}$, proving $(\bsy{a},\bsy{b})\in\PF^{(2)}_{p,q}(\bsy{U_0'})$. Reversing the argument establishes the equivalence in the other direction. Since these bounds are independent of the labelings on the edges, the corresponding claim holds for $\IPPF(p,q)$, as well.
  \end{proof}
   
\begin{example}\label{ex:pq prime}
For $(p,q)=(3,4)$, the pair $\bsy{a}=(0,0,3)$ and $\bsy{b}=(0,0,1,1)$ is a prime $(p,q)$-parking function, and can also be seen as a two-dimensional $\bsy{U_0'}$-parking function as defined in \cref{prop:2vector}. Consider the path $P=NENENNE\in\Lc(3,4)$, which is shown below. The weights according to $\bsy{U'_0}$ on its horizontal and vertical edges are, respectively, $(1,2,4)$ and $(1,1,2,2)$, which indeed bound $(\bsy{a},\bsy{b})$. We remark that the path $P$ must be such that with the exception of its first edge, it does not share any vertical edges with $L_{\bsy{b}}$ or horizontal edges with $L^\perp_{\bsy{a}}$. In particular, this can only be achieved when $L_{\bsy{b}}$ and $L^\perp_{\bsy{a}}$ do not touch except at the points $(0,0)$ and $(p,q)$.
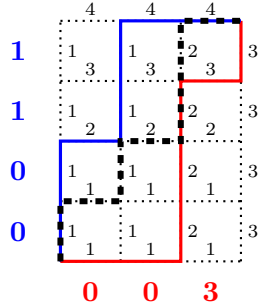
\begin{center}
\begin{tikzpicture}[scale=0.8]
\draw[step=1cm, thick, dotted] (0,0) grid (3,4);
\node at (.5, -.5) {\tcr{$\mathbf{0}$}}; 
\node at (1.5,-.5) {\tcr{$\mathbf{0}$}}; 
\node at (2.5, -.5) {\tcr{$\mathbf{3}$}}; 
\node at (-0.7, 0.5) {\tcb{$\mathbf{0}$}}; 
\node at (-0.7, 1.5) {\tcb{$\mathbf{0}$}}; 
\node at (-0.7, 2.5) {\tcb{$\mathbf{1}$}}; 
\node at (-0.7, 3.5) {\tcb{$\mathbf{1}$}}; 

\foreach \i in {0,...,3}
{
\foreach \j in {1,...,3}
{
\node at (\j+.2,\i+.5) {\tiny \j};
}
\node at (.2,\i+.5) {\tiny 1};
}
\foreach \j in {0,...,2}
{
\foreach \i in {1,...,4}
{
\node at (\j+.5,\i+.2) {\tiny \i};
}
\node at (\j+.5,.2) {\tiny 1};
}

\draw[red, very thick] (0,0)--(2,0)--(2,2)--(2,3)--(3,3)--(3,4); 
\draw[blue, very thick] (0,0)--(0,2)--(1,2)--(1,4)--(3,4);

\draw[black,line width=.8mm,dashed] (0,0)--(0,1)--(1,1)--(1,2)--(2,2)--(2,4)--(3,4);
\end{tikzpicture}
\end{center} 
\end{example}

\subsection{Decomposition into prime components}
As mentioned earlier, the notion of primeness requires that any object can be uniquely decomposed into prime components, such that the prime components are themselves indecomposable. We give an explicit description of such a decomposition for $(p,q)$-parking functions into their prime components. We first give a description for the increasing $(p,q)$-parking functions.

Let $(\bsy{i^{(1)}},\bsy{j^{(1)}}) \in IPF(p_1,q_1), \ldots, (\bsy{i^{(k)}},\bsy{j^{(k)}}) \in IPF(p_k,q_k)$ be a sequence of increasing prime $(p_\ell,q_\ell)$-parking functions for $p_1+\cdots+p_k=p$ and $q_1+\cdots+q_k=q$. Then we define the direct sum $(\bsy{i},\bsy{j})=(\bsy{i^{(1)}},\bsy{j^{(1)}}) \oplus \cdots \oplus (\bsy{i^{(k)}},\bsy{j^{(k)}})$ to be the increasing $(p,q)$-parking function corresponding to the lattice paths $L^\perp_{\bsy{i}}:=L^\perp_{\bsy{i^{(1)}}} \oplus \cdots \oplus L^\perp_{\bsy{i^{(k)}}}$ and $L_{\bsy{j}}:=L_{\bsy{j^{(1)}}} \oplus \cdots \oplus L_{\bsy{j^{(k)}}}$.

Then, any increasing $(p,q)$-parking function $(\bsy{i},\bsy{j})$ may be uniquely decomposed as a direct sum of increasing prime $(p_i,q_i)$-parking functions, for uniquely determined $(p_\ell,q_\ell)$ for $\ell=1,\ldots,k$, as follows. Let $(0,0) = (x_0,y_0), (x_1,y_1),\ldots,(x_k,y_k) = (p,q)$ denote the (pointwise increasing) coordinates of the lattice points at which $L^\perp_{\bsy{i}}$ and $L_{\bsy{j}}$ intersect. For each $1 \le \ell \le k$, let 
\begin{align*}
\bsy{i_\ell}& = (0 = a_{y_{\ell-1}} - a_{y_{\ell-1}},a_{y_{\ell-1}+1} - a_{y_{\ell-1}},\ldots,a_{y_\ell-1} - a_{y_{\ell-1}}) \text{ and} \\
\bsy{j_\ell} &= (0 = b_{x_{\ell-1}} - b_{x_{\ell-1}},b_{x_{\ell-1}+1} - b_{x_{\ell-1}},\ldots,b_{x_\ell-1} - b_{x_{\ell-1}}).
\end{align*}
If $x_{\ell-1} = x_\ell$ (respectively $y_{\ell-1} = y_\ell$), then we let $\bsy{j_\ell} = \emptyset$ (respectively $\bsy{i_\ell} = \emptyset$) by convention.

Then each $(\bsy{i_\ell},\bsy{j_\ell})$ is an increasing prime $(x_\ell-x_{\ell-1},y_\ell-y_{\ell-1})$-parking function, whose corresponding lattice paths are given by the portions of $L^\perp_{\bsy{i}}$ and $L_{\bsy{j}}$ (respectively) which lie weakly between the points $(x_{\ell-1},y_{\ell-1})$ and $(x_\ell,y_\ell)$. By construction, we thus have the following (unique) decomposition of $(\bsy{i},\bsy{j})$ into increasing prime parking functions:
\[
(\bsy{i},\bsy{j}) = (\bsy{i^{(1)}},\bsy{j^{(1)}}) \oplus \cdots \oplus (\bsy{i^{(k)}},\bsy{j^{(k)}}).
\]

\begin{example}\label{ex:pq-prime-decomp}
We show the prime decomposition of the increasing $(6,5)$-parking function $(\bsy{i},\bsy{j})$ with $\bsy{i}=(0,0,2,3,3,3)$ and $\bsy{j}=(0,0,1,5,6)$. For this example, ignore the labels on the edges and the sequences $\bsy{a}$ and $\bsy{b}$. 

\begin{center}
\begin{tikzpicture}[scale=0.6]
\draw[step=1cm, thick, dotted] (0,0) grid (6,5);

\node[purple] at (0.5,-0.3) {\tiny 1};
\node[purple] at (1.5,-0.3) {\tiny 5};
\node[purple] at (2.5,1.7) {\tiny 3};
\node[purple] at (3.5,2.7) {\tiny 0};
\node[purple] at (4.5,2.7) {\tiny 2};
\node[purple] at (5.5,2.7) {\tiny 4};

\node[cyan!50!blue] at (-0.3,0.5) {\tiny 2};
\node[cyan!50!blue] at (-0.3,1.5) {\tiny 4};
\node[cyan!50!blue] at (0.7,2.5) {\tiny 1};
\node[cyan!50!blue] at (4.7,3.5) {\tiny 3};
\node[cyan!50!blue] at (5.7,4.5) {\tiny 0};

\draw[blue,very thick]  (0,0)--(0,2)--(1,2)--(1,3)--(3,3);
\draw[blue, very thick, yshift=0.03cm] (3,3)--(5,3);
\draw[blue, very thick] (5,3)--(5,4)--(6,4);
\draw[blue, very thick, xshift=-0.03cm] (6,4)--(6,5);

\draw[red, very thick] (0,0)--(2,0)--(2,2)--(3,2)--(3,3);
\draw[red, very thick, yshift=-0.03cm] (3,3)--(5,3);
\draw[red, very thick] (5,3)--(6,3)--(6,4);
\draw[red, very thick, xshift=0.03cm] (6,4)--(6,5);

\foreach \u\i in {0/0,3/3,4/3,5/3,6/4,6/5}
{
\filldraw[red,very thick] (\u,\i) circle (3pt);
}

\node at (7.5,2.5) {$=$};
\node at (13.5,2.5) {$\oplus$};
\node at (17,2.5) {$\oplus$};
\node at (20,2.5) {$\oplus$};
\node at (23.5,2.5) {$\oplus$};


\begin{scope}[xshift=9cm,yshift=1cm]
\draw[step=1cm, thick, dotted] (0,0) grid (3,3);

\draw[red, very thick] (0,0)--(2,0)--(2,2)--(3,2)--(3,3); 
\draw[blue, very thick] (0,0)--(0,2)--(1,2)--(1,3)--(3,3);

\node[purple] at (0.5,-0.3) {\tiny 0};
\node[purple] at (1.5,-0.3) {\tiny 2};
\node[purple] at (2.5,1.7) {\tiny 1};

\node[cyan!50!blue] at (-0.3,0.5) {\tiny 1};
\node[cyan!50!blue] at (-0.3,1.5) {\tiny 2};
\node[cyan!50!blue] at (0.7,2.5) {\tiny 0};

\foreach \u\i in {0/0,3/3}
{
\filldraw[red,very thick] (\u,\i) circle (3pt);
}
\end{scope}


\begin{scope}[xshift=14.75cm,yshift=2.5cm]
\draw[step=1cm, thick, dotted] (0,0) grid (1,0);
\draw[red, very thick,yshift=-0.03cm] (0,0)--(1,0);
\draw[blue, very thick,yshift=0.03cm] (0,0)--(1,0);

\node[purple] at (0.5,-0.3) {\tiny 0};
\foreach \u\i in {0/0,1/0}
{
\filldraw[red,very thick] (\u,\i) circle (3pt);
}
\end{scope}


\begin{scope}[xshift=18cm,yshift=2.5cm]
\draw[step=1cm, thick, dotted] (0,0) grid (1,0);
\draw[red, very thick,yshift=-0.03cm] (0,0)--(1,0);
\draw[blue, very thick,yshift=0.03cm] (0,0)--(1,0);

\node[purple] at (0.5,-0.3) {\tiny 0};
\foreach \u\i in {0/0,1/0}
{
\filldraw[red,very thick] (\u,\i) circle (3pt);
}
\end{scope}


\begin{scope}[xshift=21.5cm,yshift=2cm]
\draw[step=1cm, thick, dotted] (0,0) grid (1,1);
\draw[red, very thick] (0,0)--(1,0)--(1,1);
\draw[blue, very thick] (0,0)--(0,1)--(1,1);

\node[purple] at (0.5,-0.3) {\tiny 0};
\node[cyan!50!blue] at (-0.3,0.5) {\tiny 0};
\foreach \u\i in {0/0,1/1}
{
\filldraw[red,very thick] (\u,\i) circle (3pt);
}
\end{scope}


\begin{scope}[xshift=25cm,yshift=2cm]
\draw[step=1cm, thick, dotted] (0,0) grid (0,1);

\draw[red, very thick,xshift=0.03cm] (0,0)--(0,1);
\draw[blue, very thick,xshift=-0.03cm] (0,0)--(0,1);

\node[cyan!50!blue] at (-0.3,0.5) {\tiny 0};
\foreach \u\i in {0/0,0/1}
{
\filldraw[red,very thick] (\u,\i) circle (3pt);
}
\end{scope}

\node[red] at (3,-2) {$\bsy{i}=(0,0,2,3,3,3)$};
\node[blue] at (2.7,-3) {$\bsy{j}=(0,0,1,5,6)$};
\node[purple] at (3,-4) {$\bsy{a}=(3,0,3,2,3,0)$};
\node[cyan!50!blue] at (2.7,-5) {$\bsy{b}=(6,1,0,5,0)$};

\node[red,xshift=5.5cm] at (1,-2) {$\bsy{i^{(1)}}=(0,0,2)$};
\node[blue,xshift=5.5cm] at (1,-3) {$\bsy{j^{(1)}}=(0,0,1)$};
\node[purple,xshift=5.5cm] at (1,-4) {$\bsy{a^{(1)}}=(0,2,0)$};
\node[cyan!50!blue,xshift=5.5cm] at (1,-5) {$\bsy{b^{(1)}}=(1,0,0)$};

\node[red,xshift=8.5cm] at (1,-2) {$\bsy{i^{(2)}}=(0)$};
\node[blue,xshift=8.35cm] at (1,-3) {$\bsy{j^{(2)}}=\emptyset$};
\node[purple,xshift=8.5cm] at (1,-4) {$\bsy{a^{(2)}}=(0)$};
\node[cyan!50!blue,xshift=8.35cm] at (1,-5) {$\bsy{b^{(2)}}=\emptyset$};

\node[red,xshift=10.5cm] at (1,-2) {$\bsy{i^{(3)}}=(0)$};
\node[blue,xshift=10.35cm] at (1,-3) {$\bsy{j^{(3)}}=\emptyset$};
\node[purple,xshift=10.5cm] at (1,-4) {$\bsy{a^{(3)}}=(0)$};
\node[cyan!50!blue,xshift=10.35cm] at (1,-5) {$\bsy{b^{(3)}}=\emptyset$};

\node[red,xshift=12.5cm] at (1,-2) {$\bsy{i^{(4)}}=(0)$};
\node[blue,xshift=12.5cm] at (1,-3) {$\bsy{j^{(4)}}=(0)$};
\node[purple,xshift=12.5cm] at (1,-4) {$\bsy{a^{(4)}}=(0)$};
\node[cyan!50!blue,xshift=12.5cm] at (1,-5) {$\bsy{b^{(4)}}=(0)$};

\node[red,xshift=14.35cm] at (1,-2) {$\bsy{i^{(5)}}=\emptyset$};
\node[blue,xshift=14.5cm] at (1,-3) {$\bsy{j^{(5)}}=(0)$};
\node[purple,xshift=14.35cm] at (1,-4) {$\bsy{a^{(5)}}=\emptyset$};
\node[cyan!50!blue,xshift=14.5cm] at (1,-5) {$\bsy{b^{(5)}}=(0)$};
\end{tikzpicture}
\end{center} 
\end{example}

Now let $(\bsy{a},\bsy{b})$ be a $(p,q)$-parking function. We may similarly decompose $(\bsy{a},\bsy{b})$ into prime components as follows. Let $\bsy{i} = (a_{(0)},\dots,a_{(p-1)})$ and $\bsy{j} = (b_{(0)},\dots,b_{(q-1)})$ denote the increasing rearrangements of $\bsy{a}$ and $\bsy{b}$, noting that $L^\perp_{\bsy{a}} = L^\perp_{\bsy{i}}$ and $L_{\bsy{b}} = L_{\bsy{j}}$ as unlabeled lattice paths. As above, let $(0,0) = (x_0,y_0),(x_1,y_1),\ldots,(x_k,y_k) = (p,q)$ denote the (pointwise increasing) coordinates of the lattice points at which the two lattice paths intersect. Then the decomposition of $(\bsy{i},\bsy{j})$ into prime components described above determines a decomposition of the lattice paths $L^\perp_{\bsy{a}}$ and $L_{\bsy{b}}$ (ignoring labels, for now):
\[
L^\perp_{\bsy{a}} = L^\perp_{\bsy{a},1} \oplus \cdots \oplus L^\perp_{\bsy{a},k}, \quad L_{\bsy{b}} = L_{\bsy{b},1} \oplus \cdots \oplus L_{\bsy{b},k}.
\]
For $1 \le \ell \le k$, let $A_\ell$ denote the set of labels which appear on the horizontal steps of $L^\perp_{\bsy{a},\ell}$, and let $B_\ell$ denote the set of labels which appear on the vertical steps of $L_{\bsy{b},\ell}$. Replace the labels in $A_\ell$ by $\{ 0, 1, \ldots, |A_\ell|-1 \}$ and replace the labels in $B_\ell$ by $\{ 0, 1, \ldots, |B_\ell|-1 \}$, maintaining the same numerical order in each case, to obtain a pair of labeled lattice paths ${L^\perp_{\bsy{a},\ell}}', {L_{\bsy{b},\ell}}'$, where the horizontal edges of ${L^{\perp}_{\bsy{a},\ell}}'$ are labeled by  $\mathbb{N}_{|A_\ell|}$ and the vertical edges of ${L_{\bsy{b},\ell}}'$ are labeled by $\mathbb{N}_{|B_\ell|}$.
Let $(\bsy{a}^{\bsy{(\ell)}},\bsy{b}^{\bsy{(\ell)}})$ denote the $(x_\ell - x_{\ell-1},y_\ell - y_{\ell-1})$-parking function corresponding to the pair of labeled lattice paths ${L^\perp_{\bsy{a},\ell}}',{L_{\bsy{b},\ell}}'$. Then, we write
\[
(\bsy{a},\bsy{b}) = (\bsy{a}^{\bsy{(1)}},\bsy{b}^{\bsy{(1)}}) \oplus \cdots \oplus (\bsy{a}^{\bsy{(k)}},\bsy{b}^{\bsy{(k)}}).
\]

\begin{example}
    Let $(\bsy{a},\bsy{b})$ be the $(6,5)$-parking function  with $\bsy{a} = (3,0,3,2,3,0)$ and $\bsy{b} = (6,1,0,5,0)$. The prime decomposition of the increasing rearrangement $(\bsy{i},\bsy{j})$ of $(\bsy{a},\bsy{b})$ is given in \Cref{ex:pq-prime-decomp}. Here we have 
    \[
        A_1 = \{ 1, 3, 5 \},\quad A_2 = \{ 0 \},\quad A_3 = \{ 2 \},\quad A_4 = \{ 4 \},\quad A_5 = \emptyset,\quad \text{and}
    \]
    \[
        B_1 = \{ 1, 2, 4 \},\quad B_2 = B_3 = \emptyset,\quad B_4 = \{ 3 \},\quad B_5 = \{ 0 \}.
    \]
    As a result, the prime components are given by 
    \[
    (\bsy{a}^{\bsy{(1)}},\bsy{b}^{\bsy{(1)}}) = ((0,2,0),  (1,0,0)),\quad (\bsy{a}^{\bsy{(2)}}, \bsy{b}^{\bsy{(2)}}) = (\bsy{a}^{\bsy{(3)}},\bsy{b}^{\bsy{(3)}}) = ((0),\emptyset),
    \]
    \[
    (\bsy{a}^{\bsy{(4)}}, \bsy{b}^{\bsy{(4)}}) = ((0),(0)),\quad (\bsy{a}^{\bsy{(5)}},\bsy{b}^{\bsy{(5)}}) = (\emptyset,(0)).
    \]
\end{example}

\subsection{Enumeration of prime $(p,q)$-parking functions}

Increasing $(p,q)$-parking functions are in bijection with pairs of lattice paths, letting us easily compute $\#\IPPF(p,q)$. 

\begin{lemma}\label{lemma:IPPF-2dim}
The number of increasing prime $(p,q)$-parking functions is 
$$
\#\IPPF(p,q)=\frac{1}{p+q-1}\binom{p+q-1}{p-1}\binom{p+q-1}{q-1}\, .$$
\end{lemma}
\begin{proof}
From \cref{prop:equivalent}, an increasing prime $(p,q)$-parking function $(\bsy{a},\bsy{b})$ must have $a_0=0$, $b_0=0$, $a_{p-1}\leq q-1$, and $b_{q-1}\leq p-1$. Thus the set of increasing prime $(p,q)$-parking functions is in bijection with pairs of non-intersecting lattice paths from $(1,0)$ to $(p,q-1)$ and $(0,1)$ to $(p-1,q)$, respectively, which can be counted using the Lindstr\"{o}m--Gessel--Viennot Lemma \cite{GV85} to obtain $\#\IPPF(p,q)=\binom{p+q-2}{q-1}^2-\binom{p+q-2}{q}\binom{p+q-2}{p}$. 
\end{proof}

To enumerate the prime $(p,q)$-parking functions, we use the equivalence given in \cref{prop:2vector}. 

\begin{lemma}\label{lem:sum} For $pq\geq 1$,
    \begin{equation}\label{eq:PFZ}
    \#\PPF(p,q)=\sum_{i=1}^p\sum_{j=1}^q \binom{p}{i}\binom{q}{j}(1+qi+pj-p-q-ij)(q-1)^{p-i-1}(p-1)^{q-j-1}.
    \end{equation}
\end{lemma}

\begin{proof}
    For a $(p,q)$-parking function $(\bsy{a},\bsy{b})$, define $m_0(\bsy{a})=\#\{j:a_j=0\}$ and $m_0(\bsy{b})=\#\{j:b_j=0\}$ to be the number of 0's in $\bsy{a}$ and $\bsy{b}$, respectively. Further, let $\bsy{a'}$ and $\bsy{b'}$ be the subsequences of positive integers of $\bsy{a}$ and $\bsy{b}$, respectively.
    For $(0,0)\leq (i,j)\leq (p,q)$, define 
    \[\PPF^{(i,j)}(p,q):=\{(\bsy{a},\bsy{b})\in\PPF(p,q):(m_0(\bsy{a}),m_0(\bsy{b}))=(i,j)\}.
    \] 
    We claim that $(\bsy{a},\bsy{b})\in \PPF^{(i,j)}(p,q)$ precisely when $(\bsy{a'}-1,\bsy{b'}-1)$ is a $\bsy{Z^{(i,j)}}$-parking function for the weight matrix $\bsy{Z^{(i,j)}}=\{ (u_{k,\ell},v_{k,\ell}):  (0,0) \leq (k,\ell) \leq (p-i,q-j)\}$ given by
    \[
    (u_{k,\ell},v_{k,\ell})=(\ell+j-1,k+i-1).
    \]
    
    Observe that the point $(i,j)$ is weakly above $L_{\bsy{a}}^{\perp}$ and weakly to the right of $L_{\bsy{b}}$, and let $P=(r_1,\ldots,r_{p+q})\in\mathcal{L}(p,q)$ be a lattice path that lies weakly between $L_{\bsy{a}}^{\perp}$ and $L_{\bsy{b}}$  and passes through $(i,j)$. 
    By \cref{prop:2vector}, $(\bsy{a},\bsy{b})$ is bounded by $P$ 
    with respect to $\bsy{Z'}$ defined by $z'_{k,\ell}=(\ell,k)$ for $k,\ell\geq 1$. Let the edge weights of $P$ with respect to $\bsy{Z'}$ be $\bsy{s}=(s_1,\ldots,s_p)$ and $\bsy{t}=(t_1,\ldots,t_q)$ for the horizontal and vertical edges, respectively, so that $(\bsy{a},\bsy{b})$ is bounded by $(\bsy{s},\bsy{t})$. 
    Let $P'=(r_{i+j+1},\ldots,r_{p+q})$ be the portion of $P$ following the point $(i,j)$. The edge weights of $P'$ with respect to $\bsy{Z^{(i,j)}}$ are $\bsy{s'}=(s_{i+1}-1,\ldots,s_{p}-1)$ and $\bsy{t'}=(t_{j+1}-1,\ldots,t_{q}-1)$. Thus $(\bsy{a'}-1,\bsy{b'}-1)$ is bounded by $(\bsy{s'},\bsy{t'})$, so $(\bsy{a'}-1,\bsy{b'}-1)\in {\PF^{(2)}_{p,q}(\bsy{Z^{(i,j)}})}$, proving our claim. 
     
Since there are $\binom{p}{i}$ and $\binom{q}{j}$ ways of arranging the $i$ 0's in $\bsy{a}$ and $j$ 0's in $\bsy{b}$, respectively,  we have
\[
\#\PPF(p,q)
=\sum_{i=1}^p\sum_{j=1}^q\#\PPF^{(i,j)}(p,q)
=\sum_{i=1}^p\sum_{j=1}^q \binom{p}{i}\binom{q}{j} \#{\PF^{(2)}_{p,q}(\bsy{Z^{(i,j)}}})\, .\]
We obtain the desired equality from \cref{eq:goncarov} by taking $\begin{pmatrix}
  a & b  \\
  c & d 
  \end{pmatrix} =\begin{pmatrix} 
  0 & 1  \\
  1 & 0 
  \end{pmatrix} $ and $\begin{pmatrix} 
  s  \\
  t 
  \end{pmatrix} =\begin{pmatrix}
  j-1  \\
  i-1 
  \end{pmatrix} $.
  
\end{proof}

\begin{theorem} \label{thm:PPF-2dim}
    When $p,q\geq 1$, the number of prime $(p,q)$-parking functions is
    \[\#\PPF(p,q)=p^q(q-1)^{p-1}+q^p(p-1)^{q-1}-(p+q-1)(p-1)^{q-1}(q-1)^{p-1}.\]
    If $p=0, q=1$ or $p=1, q=0$, $\#\PPF(p,q)=1$, and in all other cases, $\#\PPF(p,q)=0$.
\end{theorem}
\begin{proof}
When $pq=0$, the result follows directly from \cref{IPPFLemma}. Let $p,q\geq 1$, and set $c_{i,j}=\binom{p}{i}\binom{q}{j}(1+qi+pj-p-q-ij)(q-1)^{p-i-1}(p-1)^{q-j-1}$. By \cref{lem:sum}, the number of prime $(p,q)$-parking functions is
\[
\sum_{i=1}^p\sum_{j=1}^q c_{i,j}=\sum_{i=0}^p\sum_{j=0}^q c_{i,j} - \sum_{i=0}^p c_{i,0}-\sum_{j=0}^q c_{0,j}+c_{0,0}\, .
\]

\noindent 
We compute  
\begin{align*}
    \sum_{j=0}^q c_{0,j}&=(q-1)^{p-1}\Big((1-p-q)\sum_{j=0}^q \binom{q}{j}(p-1)^{q-j-1}+p\sum_{j=0}^q\binom{q}{j}j(p-1)^{q-j-1}\Big)\\
&=(q-1)^{p-1}(p-1)^{-1}\Big((1-p-q)p^q+qp^q\Big)
\\
&=-p^q(q-1)^{p-1},
\end{align*}
and by symmetry, $\sum_{i=0}^p c_{i,0}=-q^p(p-1)^{q-1}$. Then we have
\begin{align*}
    \sum_{i=0}^p\sum_{j=0}^p c_{i,j}&=\sum_{i=0}^p \binom{p}{i}(q-1)^{-i} \sum_{j=0}^q c_{0,j}+ \sum_{i=0}^p i \binom{p}{i}(q-1)^{p-i-1}\sum_{j=0}^q (q-j)\binom{q}{j}(p-1)^{q-j-1}\\
    &=-\sum_{i=0}^p \binom{p}{i}(q-1)^{-i} p^q(q-1)^{p-1}+\sum_{i=0}^p i \binom{p}{ i}(q-1)^{p-i-1}qp^{q-1}\\
    &=0\, .
\end{align*}
In the computations, we use the identities $\sum_{i=0}^a \binom{a}{i}x^i=\sum_{i=0}^a \binom{a}{ i}x^{a-i}=(1+x)^a$ and $\sum_{i=0}^ai \binom{a}{i}x^{i-1}=\sum_{i=0}^a(a-i) \binom{a}{i}x^{a-i-1}=a(1+x)^{a-1}$.

Combining the above with $c_{0,0}=(1-p-q)(q-1)^{p-1}(p-1)^{q-1}$, we get the desired equality.
\end{proof}

\subsection{Generalization for two-dimensional vector parking functions}\label{subsection:generalization for two dim}

Inspired by the primeness of $(p,q)$-parking functions, we propose a definition of primeness for two-dimensional vector 
parking functions with weight matrix $\bsy{U}$ for $p,q\geq 1$. 

\begin{definition} \label{def:general}
Let $p,q$ be positive integers. 
Assume $\bsy{U}=\{ (u_{k,\ell}, v_{k,\ell}): (0,0) \leq (k,\ell) \leq (p,q)\} \subset \mathbb{N}^2$. 
Let $\bsy{a}=(a_0, \dots, a_{p-1})\in\mathbb{N}^p$  and $\bsy{b}=(b_0, \dots, b_{q-1})\in\mathbb{N}^q$. 
 A pair  $(\bsy{a}, \bsy{b})$  is a \emph{prime two-dimensional $\bsy{U}$-parking function} if and only if 
there are lattice paths $P_1, P_2 \in \mathcal{L}(p,q)$   such that 
\begin{enumerate}
    \item The order statistics of $(\bsy{a}, \bsy{b})$ are bounded by both $P_1$ and $P_2$ with respect to $\bsy{U}$, and
    \item The only lattice points at which the lattice paths $P_1$ and $P_2$ intersect are 
    $(0,0)$ and $(p,q)$. 
\end{enumerate}
Denote by $\PPF^{(2)}_{p,q}(\bsy{U})$ the set of prime two-dimensional $\bsy{U}$-parking functions and by $\IPPF^{(2)}_{p,q}(\bsy{U})$ the set of
prime increasing ones.  
\end{definition}

Given $\bsy{U}=\{z_{k,\ell}=(u_{k,\ell}, v_{k,\ell}): \ (0,0) \leq (k,\ell) \leq (p,q)\}$, define a new weight matrix $\bsy{U'}$ by letting 
$\bsy{U'}=\{z'_{k,\ell}=(u'_{k,\ell}, v'_{k,\ell}): \ (0,0) \leq (k,\ell) \leq (p,q)\}$, where 
\begin{eqnarray}
    z'_{k,\ell} =(u'_{k,\ell}, v'_{k,\ell})=  \left\{ \begin{array}{ll} 
         (u_{k,\ell-1}, v_{k-1,\ell}) & \text{ if } k,\ell \geq 1, \\
        (u_{k,0}, v_{0,0})      & \text{if }  \ell=0  \\ 
        (u_{0,0}, v_{0,\ell})      & \text{ if }  k=0. \end{array} \right. 
 \end{eqnarray}
In particular, $z'_{0,0}=(u_{0,0}, v_{0,0})$. 

\begin{theorem}\label{Theorem:affine-2-dim}
The following sets are equal 
    \[
    PPF^{(2)}_{p,q}(\bsy{U}) =  PF^{(2)}_{p,q}(\bsy{U')}, 
    \]
    and 
    \[
     IPPF^{(2)}_{p,q}(\bsy{U}) =  IPF^{(2)}_{p,q}(\bsy{U'}). 
    \]
\end{theorem}
\begin{proof}
Let $(\bsy{a}, \bsy{b}) \in \mathbb{N}^p \times \mathbb{N}^q$ be a prime two-dimensional $\bsy{U}$-parking function, with lattice paths $P_1$ and $P_2$ with $P_2$ to the left of $P_1$, satisfying the conditions of \cref{def:general}. One notices that $P_1$ must start with  $E$ and end with  $N$, and vice versa for $P_2$. Let $P_1'$ be the part of $P_1$ from $(1,0)$ to $(p, q-1)$ and let $P_2'$ be the part of $P_2$ from $(0,1)$ to $(p-1, q)$. See \cref{fig:general}(a) for an example. 

Let $Q_1$ be the path obtained from $P_1'$ by moving one unit up, and $Q_2$ the path obtained from $P_2'$ by moving one unit to the right. 
Then $Q_1$ and $Q_2$ are lattice paths from $(1,1)$ to $(p,q)$ with $Q_2$ weakly to the left of $Q_1$. See \cref{fig:general}(b) for an example. 

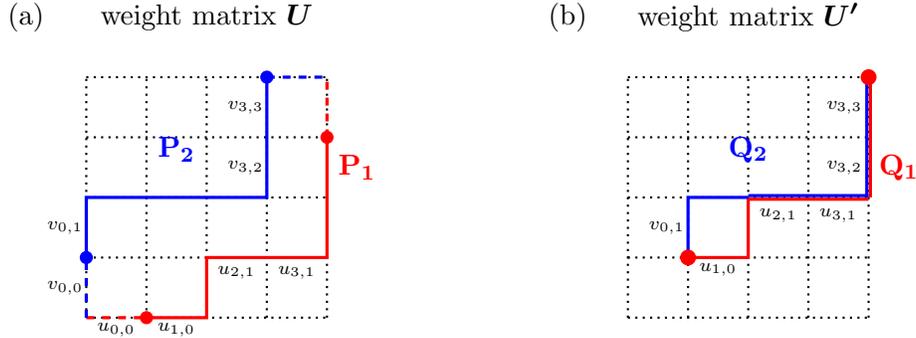
\begin{figure}[h] 
\begin{center}
\begin{tikzpicture}[scale=0.8]
\node at (-1,5) {(a)};
\draw[step=1cm, thick, dotted] (0,0) grid (4,4);

\node at (4.5, 2.5) {\tcr{$\mathbf{P_1}$}}; 
\node at (1.5,2.8) {\tcb{$\mathbf{P_2}$}}; 
\node at (2,5) {weight matrix $\bsy{U}$}; 

\draw[red,very thick, dashed] (0,0)--(1,0);
\draw[red,very thick, dashed] (4,3)--(4,4);
\draw[red, very thick] (1,0)--(2,0)--(2,1)--(4,1)--(4,3); 
\draw[blue,very thick, dashed] (0,0)--(0,1);
\draw[blue,very thick, dashed] (3,4)--(4,4);
\draw[blue, very thick] (0,1)--(0,2)--(0,2)--(3,2)--(3,4); 

\filldraw[color=red] (1,0) circle [radius=0.1];
\filldraw[color=red] (4,3) circle [radius=0.1];
\filldraw[color=blue] (0,1) circle [radius=0.1];
\filldraw[color=blue] (3,4) circle [radius=0.1];

\node at (0.5,-0.2)  {\tiny{$u_{0,0}$}}; 
\node at (1.5, -0.2) {\tiny{$u_{1,0}$}}; 
\node at (2.5, .75) {\tiny{$u_{2,1}$}}; 
\node at (3.5, .75) {\tiny{$u_{3,1}$}}; 

\node at (-0.35,0.5) {\tiny{$v_{0,0}$}};
\node at (-0.35,1.5) {\tiny{$v_{0,1}$}};
\node at (2.65,2.5) {\tiny{$v_{3,2}$}};
\node at (2.65,3.5) {\tiny{$v_{3,3}$}};

\node at (8,5) {(b)};
\begin{scope}[xshift=9cm,yshift=0cm]
\draw[step=1cm, thick, dotted] (0,0) grid (4,4);

\node at (4.5, 2.5) {\tcr{$\mathbf{Q_1}$}}; 
\node at (2,2.8) {\tcb{$\mathbf{Q_2}$}}; 
\node at (2,5) {weight matrix $\bsy{U'}$}; 

\draw[red, very thick] (1,1)--(2,1)--(2,2); 
\draw[red, very thick, yshift=-0.03cm] (2,2)--(4,2);
\draw[red,very thick, xshift=0.03cm] (4,2)--(4,4); 
\draw[blue, very thick] (1,1)--(1,2)--(2,2);
\draw[blue,very thick, yshift=0.03cm] (2,2)--(4,2); 
\draw[blue,very thick, xshift=-0.03cm] (4,2)--(4,4); 

\node at (1.5, .8) {\tiny{$u_{1,0}$}}; 
\node at (2.5, 1.7) {\tiny{$u_{2,1}$}}; 
\node at (3.5, 1.7) {\tiny{$u_{3,1}$}}; 

\node at (0.65,1.5) {\tiny{$v_{0,1}$}};
\node at (3.6,2.5) {\tiny{$v_{3,2}$}};
\node at (3.6,3.5) {\tiny{$v_{3,3}$}};

\filldraw[red,very thick] (1,1) circle (3pt);
\filldraw[red,very thick] (4,4) circle (3pt);

\end{scope}
\end{tikzpicture}
\end{center} 
\caption{(a) We show lattice paths $(P_1,P_2)$ with weights given by $\bsy{U}$ with $(P_1', P_2')$ the portions of $(P_1,P_2)$ marked by solid lines. (b) $(P_1',P_2')$ correspond to lattice paths $(Q_1,Q_2)$, and the weights on $(P'_1,P'_2)$ according to $\bsy{U}$ induce the weights on $(Q_1,Q_2)$ according to $\bsy{U'}$ on the 
sub-grid $\{(k,\ell): (1,1) \leq (k,\ell) \leq (p,q) \}$. }
\label{fig:general}
\end{figure}

By definition, the order statistics of $(\bsy{a}, \bsy{b})$ are bounded by both $P_1$ and $P_2$ with respect to $\bsy{U}$. Hence 
$a_{(0)} < u_{0,0}$ and $b_{(0)} < v_{0,0}$.
In addition, the order statistics of  $\bsy{a}\setminus a_{(0)}$ is bounded piecewise by the weight of horizontal steps of $P_1'$,
or equivalently, the weight of horizontal steps of $Q_1$ with respect to the weight $\bsy{U'}$. Similar statements hold for $\bsy{b}\setminus b_{(0)}$ 
and the vertical steps of $Q_2$. Let $P=\{EN\}\oplus Q$ where $Q$ is any lattice path lying between 
$Q_1$ and $Q_2$. Then $(\bsy{a}, \bsy{b})$ is a 
two-dimensional $\bsy{U'}$-parking function
whose order statistics are bounded by $P$ with respect to $\bsy{U'}$. 

Conversely, let $(\bsy{a}, \bsy{b})$ be  a 
two-dimensional $\bsy{U'}$-parking function
 that is bounded by a lattice path $Q$ with respect to $\bsy{U'}$. Without loss of generality, assume $Q$ starts with $E$ and the initial segment of $Q$ is $E^kN$ for some $k \geq 1$.  Let $Q'$ be obtained from $Q$ by replacing  the initial $E^kN$ with $ENE^{k-1}$. Then $Q'$ passes through $(1,1)$, and its edge weights bound 
 the order statistics of $(\bsy{a}, \bsy{b})$ as well. Now moving back to the grid with the weight matrix $\bsy{U}$. Let $R$ be the part of $Q'$ from $(1,1)$ to $(p,q)$, $R_1$ obtained from $R$ by moving one unit down, and 
 $R_2$ obtained from $R$ by moving one unit to the left.  Finally let $P_1=\{E\}\oplus R_1 \oplus \{N\}$ and $P_2=\{N\}\oplus R_2\oplus \{E\}$. Then $P_1$ and $P_2$ are non-touching lattice paths whose weights  bound
 $(\bsy{a}, \bsy{b})$ 
 with respect to the weight $\bsy{U}$, as required by \cref{def:general}. 
 
\end{proof}

Finally, we present explicit formulas for  the numbers of prime two-dimensional $\bsy{U}$-parking functions 
and their increasing analogs for the  affine case, that is, $\bsy{U}$ is given by the following equation:    
\begin{equation} \label{Weight:affine}
\left(\begin{array}{c} 
u_{k,\ell} \\ 
v_{k,\ell} 
\end{array} \right) =\left( \begin{array}{cc} 
  a & b  \\
  c & d 
  \end{array} 
  \right)
  \left(\begin{array}{c} 
k \\ 
\ell
\end{array} \right) + 
\left(\begin{array}{c} 
s \\ 
t
\end{array} \right),
\end{equation} 
with $a, b, c, d, s, t \in \mathbb{N}$.

\begin{theorem} \label{thm:general}
For the general  affine weight $\bsy{U}$ given by \cref{Weight:affine}, let $X=ap+b(q-1)$ and $Y=c(p-1)+dq$. Then 
\begin{multline*} 
    \#\IPPF^{(2)}_{p,q}(\bsy{U})=  \frac{1}{p!q!}\big[\big( (s+bq-b)(t+cp-c)-bcpq\big) (s+X+1)^{(p-1)} (t+Y+1)^{(q-1)} \\
     +   b\big(c(p+q-1)+t-tq  \big) (X+1)^{(p-1)} (t+Y+1)^{(q-1)} \\ 
     +  c \big(b(p+q-1)+s-sp  \big)(s+X+1)^{(p-1)} (Y+1)^{(q-1)} \\
      - bc(p+q-1) (X+1)^{(p-1)} (Y+1)^{(q-1)} \big] 
    \end{multline*}   
and 
\begin{multline*} 
    \#\PPF^{(2)}_{p,q}(\bsy{U})=  \big( (s+bq-b)(t+cp-c)-bcpq\big) (s+X)^{p-1} (t+Y)^{q-1} 
     +   b\big(c(p+q-1)+t-tq  \big) X^{p-1} (t+Y)^{q-1} \\ 
     +  c \big(b(p+q-1)+s-sp  \big)(s+X)^{p-1} Y^{q-1}
      - bc(p+q-1) X^{p-1} Y^{q-1}.
    \end{multline*}
\end{theorem}

Note that when $a=d=0$ and $b=c=s=t=1$, we have $X=q-1$ and $Y=p-1$, and the formulas in \cref{thm:general} agree with those in \cref{lemma:IPPF-2dim} and \cref{thm:PPF-2dim}. 

\begin{proof}
By \cref{Theorem:affine-2-dim}, it suffices to compute the number of two-dimensional  $\bsy{U'}$-parking functions. Assume $(\bsy{a}, \bsy{b}) \in \mathbb{N}^p \times \mathbb{N}^q$ is  a $\bsy{U'}$-parking function.  
  Let $\bsy{a}_1$ and $\bsy{a}_2$ be the subsequences of $\bsy{a}$ consisting of entries less than $u_{0,0}=s$ and greater than or equal to $s$, respectively.  Similarly, let $\bsy{b}_1$ and $\bsy{b}_2$ be the subsequences of $\bsy{b}$  consisting of entries respectively less than $v_{0,0}=t$ and greater than or equal to $t$. Let $i, j$ be the length of $\bsy{a}_1$ and $\bsy{b}_1$, respectively.

We claim that if $P$ is a lattice path whose weight sequence bounds $(\bsy{a}, \bsy{b})$,  then there is a lattice path $P'$ passing through $(i,j)$ whose weight sequence also bounds $(\bsy{a}, \bsy{b})$ (both with respect to $\bsy{U'}$). Indeed, $P'$ can be constructed as follows. Let $i'$ and $j'$ be maximal such that the points $(i',j)$ and $(i,j')$ are on $P$. Then either $i'<i$ and $j<j'$, or $i'>i$ and $j>j'$; without loss of generality, assume the former. Let $R$ be the portion of $P$ after the point $(i,j')$. Then we let $P'=\{E^{i}N^{j'}\}\oplus R$. We have that the weight sequence on the vertical steps of $P_1$ is at least that of $P$, and the weights on the first $i$ horizontal steps of $P'$ are all $s$, which bound the first $k$ order statistics of $\bsy{a}$. Thus $P'$ bounds $(\bsy{a},\bsy{b})$.

Define the weight matrix $\bsy{U^{(i,j)}}=\{ (z^{(i,j)}_{k,\ell}: 0 \leq k \leq p-i, 0 \leq \ell \leq q-j \}$, where 
    \begin{equation*}
z^{(i,j)}_{k,\ell} =\left( \begin{array}{cc} 
  a & b  \\
  c & d 
  \end{array} 
  \right)
  \left(\begin{array}{c} 
k \\ 
\ell
\end{array} \right) + 
\left(\begin{array}{c} 
ai+b(j-1) \\ 
c(i-1)+dj 
\end{array} \right). 
\end{equation*} 
Let $P$ be a lattice path that passes through the point $(i,j)$, and consider its horizontal and vertical weight sequences $(h_0,\ldots,h_{p-1})$ and $(v_0,\ldots,v_{q-1})$, respectively. Let $R$ be the portion of $P$ after the point $(i,j)$, and let $R'$ be the path $R$ shifted to start at the origin $(0,0)$. Then the horizontal and vertical weight sequences of $R'$ with respect to $\bsy{U^{(i,j)}}$ are precisely $(h_{i},\ldots,h_{p-1})-s$ and $(v_{j},\ldots,v_{q-1})-t$, respectively. Since $R$ bounds the portion of $(\bsy{a},\bsy{b})$ corresponding to $(\bsy{a}_2,\bsy{b}_2)$ with respect to $\bsy{U'}$, we have that $R'$ bounds $(\bsy{a}_2-s,\bsy{b}_2-t)$ with respect to $\bsy{U^{(i,j)}}$.

Conversely, suppose $(\bsy{a}'_2,\bsy{b}'_2)\in \mathbb{N}^{p-i}\times\mathbb{N}^{q-j}$ is a pair of sequences bounded by some path $R$ with respect to $\bsy{U^{(i,j)}}$. For some $(\bsy{a}_1,\bsy{b}_2)\in\mathbb{N}_s^i\times\mathbb{N}_t^{j}$, define $\bsy{a}$ to be a shuffle of $\bsy{a}_1$ and $\bsy{a}_2'+s$, and $\bsy{b}$ to be a shuffle of $\bsy{b}_1$ and $\bsy{b}_2'+t$. Then we have that $(\bsy{a},\bsy{b})$ is bounded by the path $\{E^iN^{j}\}\oplus R$ with respect to $\bsy{U}'$.

Thus we may conclude that the pair $(\bsy{a}_2-s, \bsy{b}_2-t)$  is exactly a two-dimensional vector parking function for the weight matrix $\bsy{U^{(i,j)}}$.

To count the number of possible sequences $\bsy{a}$ corresponding to a fixed $\bsy{a}_2$, we observe that $\bsy{a}_1\in\mathbb{N}_s^i$, and that $\bsy{a}$ is a shuffle of $\bsy{a}_1$ and $\bsy{a}_2$. Thus there are $\binom{p}{i}s^i$ sequences $\bsy{a}$ for each $\bsy{a}_2$. Similarly, there are $\binom{q}{ j}t^{j}$ possible sequences $\bsy{b}$ for each $\bsy{b}_2$. Combining the above with  \cref{eq:goncarov}, 
\begin{eqnarray}
    \# \PPF^{(2)}_{p,q}(\bsy{U}) & = &\sum_{i=1}^p \sum_{j=1}^q \binom{p}{i} \binom{q}{j} s^i t^{j} \#\PF^{(2)}_{p-i, q-j}(\bsy{U^{(i,j)}}) 
 \nonumber \\
    & = & \sum_{i=1}^p \sum_{j=1}^q \binom{p}{i} \binom{q}{j} s^i t^{j}  A  X^{p-i-1} Y^{q-j-1} 
\end{eqnarray}
where 
\begin{eqnarray*}
X=ap+b(q-1),   \qquad  
    Y=c(p-1)+dq, 
    \end{eqnarray*} 
    and 
    \begin{eqnarray*}
    A&=&ij(ad-bc) + ic\big( a(p-1)+bq\big) +bj \big(cp+d(q-1)\big) +bc(1-p-q).       
\end{eqnarray*}

If we restrict to counting pairs of nondecreasing sequences $(\bsy{a},\bsy{b})$ corresponding to a nondecreasing pair $(\bsy{a}_2,\bsy{b}_2)$, we have that there are $\binom{s+i-1}{ i}=s^{(i)}/i!$ choices for $\bsy{a}_1\in\mathbb{N}_s^i$ and $\binom{t+j-1}{ j}=t^{(j)}/j!$ choices for $\bsy{b}_1\in\mathbb{N}_t^{j}$ (where $x^{(n)}=x(x+1) \cdots (x+n-1)$ is the rising factorial). Combining the above with \cref{eq:in-goncarov},  

\begin{eqnarray} \label{eq:IPPF-2-dim}
    & &  \#IPPF^{(2)}_{p,q}(\bsy{U})  =  \sum_{i=1}^p \sum_{j=1}^q \frac{A}{i!j!} s^{(i)} t^{(j)} 
     \# IPF^{(2)}_{p-i,q-j}(\bsy{U^{(i,j)}}) \nonumber \\
     & = &  \sum_{i=1}^p \sum_{j=1}^q \frac{A}{(p-i)(q-j) i!j!} s^{(i)} t^{(j)} 
     \binom{ap+b(q-1)+p-i-1}{p-i-1} \binom{c(p-1)+dq+q-j-1}{q-j-1}    \nonumber \\    
     & = & \frac{1}{XY} \sum_{i=1}^p \sum_{j=1}^q  A \binom{s+i-1}{i} \binom{t+j-1}{j}  \binom{X+p-i-1}{p-i} \binom{Y+q-j-1}{q-j}  \\
     & = & \frac{1}{p!q! XY} \sum_{i=1}^p \sum_{j=1}^q \binom{p}{i} \binom{q}{j} A s^{(i)} t^{(j)}  X^{(p-i)} Y^{(q-j)}. 
\end{eqnarray}

Using binomial identities for $x^n$ and the rising factorial $x^{(n)}$ and a similar computation as in the proof of \cref{thm:PPF-2dim}, we obtain the desired formulas. 
\end{proof}

\section{Final remarks and open problems}  \label{section:final}

Classical parking functions are well-known for their connections to various discrete and algebraic structures. Initially introduced as hashing functions to analyze the distribution of linear probes, parking functions have since been linked to numerous areas, including the enumeration of labeled trees and forests, hyperplane arrangements, noncrossing partition lattices, monomial ideals, and the combinatorial theory of Macdonald polynomials, among others \cite{Yan2015}. However, the properties of prime parking functions remain relatively unexplored. A deeper understanding of the roles that prime parking functions and their sum-enumerators play in combinatorics and algebra would be highly valuable.

There are numerous generalizations of classical parking functions in the literature. In this paper, we explore the concept of primeness in the context of 
vector parking functions, $(p,q)$-parking functions, and two-dimensional vector parking functions. 
For the first two cases, the corresponding parking functions can be represented by labeled lattice paths, where the notion of an ``indecomposable component" is well-defined in terms of the lattice paths, guiding an appropriate definition of primeness. We extend this notion to two-dimensional vector parking functions. However, the interpretation of ``indecomposable" in this latter case remains unclear.

For dimension $d\geq 3$, there are notions of $(p_1,p_2,\dots, p_d)$-parking functions  (as defined by Cori and {Poulalhon}) and $d$-dimensional vector parking functions.  
However, in both cases, we lack an interpretation involving disjoint lattice paths. Extending the concept of primeness to higher-dimensional parking functions remains an open question.

Beyond the generalized parking functions discussed in this paper, there are other models, such as $G$-parking functions, which are linked to the recurrent configurations of sandpile models on a digraph $G$ \cite{Psh2004}. Other examples include parking functions and mappings defined through a parking process on trees or digraphs\cite{BGY2017,LP2016}, parking functions involving cars of different sizes \cite{AJC79,EH2016}, or parking processes allowing backward movement \cite{CHJ+2020}. While there is extensive literature on these topics, to our knowledge, primeness has only been explored for parking functions on trees (see \cite{KY2019}). Much remains to be investigated in this area.

\section*{Acknowledgements} 

This material is based upon work supported by the National Science Foundation under Grant Number DMS 1916439 while the authors were  participating in the 
Mathematics Research Communities (MRC) 2024 Summer Conference at Beaver Hollow Conference Center in Java Center, NY. 
 We thank AMS for facilitating
and supporting the MRC. 

We thank Collier Gaiser for helpful discussions during the MRC week. Mandelshtam is supported by NSERC grant RGPIN-2021-02568. Martinez is supported by the NSF Graduate Research Fellowship Program under Grant No. 2233066. Yan is supported in part by the Simons Collaboration
Grant for Mathematics 704276.

\bibliographystyle{amsplain}
\bibliography{biblioPrime}
\end{document}